\documentclass[a4paper,11pt]{article}
\usepackage{amsfonts,amssymb,amsmath, dsfont, graphicx}
\usepackage{bm}
\usepackage{mathrsfs}  
\usepackage{color}
\usepackage{upgreek}
\usepackage[english]{babel}
\usepackage{graphicx}
\usepackage{microtype}
\usepackage[colorlinks=true, pdfstartview=FitV, linkcolor=blue, citecolor=blue, urlcolor=blue,pagebackref=false]{hyperref}

\definecolor{darkred}{rgb}{0.9,0.1,0.1}

 \makeatletter
 \@addtoreset{equation}{section}
 \makeatother

 \makeatletter
 \@addtoreset{enunciato}{section}
 \makeatother

 \newcounter{enunciato}[section]

 \newtheorem{ittheorem}{Theorem}
 \newtheorem{itlemma}{Lemma}
 \newtheorem{itproposition}{Proposition}
 \newtheorem{itcorollary}{Corollary}
 \newtheorem{itdefinition}{Definition}
 \newtheorem{itremark}{Remark}
 \newtheorem{itclaim}{Claim}
 \newtheorem{itfact}{Fact}
 \newtheorem{itconjecture}{Conjecture}

 \newenvironment{theorem}{\addtocounter{enunciato}{1}
 \begin{ittheorem}}{\end{ittheorem}}

 \newenvironment{lemma}{\addtocounter{enunciato}{1}
 \begin{itlemma}}{\end{itlemma}}

 \newenvironment{proposition}{\addtocounter{enunciato}{1}
 \begin{itproposition}}{\end{itproposition}}

 \newenvironment{corollary}{\addtocounter{enunciato}{1}
 \begin{itcorollary}}{\end{itcorollary}}

 \newenvironment{definition}{\addtocounter{enunciato}{1}
 \begin{itdefinition}}{\end{itdefinition}}

 \newenvironment{remark}{\addtocounter{enunciato}{1}
 \begin{itremark}}{\end{itremark}}

 \newenvironment{claim}{\addtocounter{enunciato}{1}
 \begin{itclaim}}{\end{itclaim}}

 \newenvironment{fact}{\addtocounter{enunciato}{1}
 \begin{itfact}}{\end{itfact}}

 \newenvironment{conjecture}{\addtocounter{enunciato}{1}
 \begin{itconjecture}}{\end{itconjecture}}

 \newcommand{\be}[1]{\begin{equation}\label{#1}}
 \newcommand{\ee}{\end{equation}}

 \newcommand{\bl}[1]{\begin{lemma}\label{#1}}
 \newcommand{\el}{\end{lemma}}

 \newcommand{\br}[1]{\begin{remark}\label{#1}}
 \newcommand{\er}{\end{remark}}

 \newcommand{\bt}[1]{\begin{theorem}\label{#1}}
 \newcommand{\et}{\end{theorem}}

 \newcommand{\bd}[1]{\begin{definition}\label{#1}}
 \newcommand{\ed}{\end{definition}}

 \newcommand{\bcl}[1]{\begin{claim}\label{#1}}
 \newcommand{\ecl}{\end{claim}}

 \newcommand{\bfact}[1]{\begin{fact}\label{#1}}
 \newcommand{\efact}{\end{fact}}

 \newcommand{\bp}[1]{\begin{proposition}\label{#1}}
 \newcommand{\ep}{\end{proposition}}

 \newcommand{\bc}[1]{\begin{corollary}\label{#1}}
 \newcommand{\ec}{\end{corollary}}

 \newcommand{\bcj}[1]{\begin{conjecture}\label{#1}}
 \newcommand{\ecj}{\end{conjecture}}

 \newcommand{\bpr}{\begin{proof}}
 \newcommand{\epr}{\end{proof}}

 \newcommand{\bprlem}[1]{\begin{proofof}{\it Lemma \ref{#1}}.\,\,}
 \newcommand{\eprlem}{\end{proofof}}

 \newcommand{\bprthm}[1]{\begin{proofof}{\it Theorem \ref{#1}}.\,\,}
 \newcommand{\eprthm}{\end{proofof}}

 \newcommand{\bprprop}[1]{\begin{proofof}{\it Proposition \ref{#1}}.\,\,}
 \newcommand{\eprprop}{\end{proofof}}

 \newcommand{\bi}{\begin{itemize}}
 \newcommand{\ei}{\end{itemize}}

 \newcommand{\ben}{\begin{enumerate}}
 \newcommand{\een}{\end{enumerate}}

 \newenvironment{proof}{\noindent {\em Proof}.\,\,}{\hspace*{\fill}$\halmos$\medskip}
 \newenvironment{proofof}{\noindent {\em Proof of\,\,}}{\hspace*{\fill}$\halmos$\medskip}
 \newcommand{\halmos}{\rule{1ex}{1.4ex}}

 \newcommand{\one}{{\mathchoice {1\mskip-4mu\mathrm l}
         {1\mskip-4mu\mathrm l}
         {1\mskip-4.5mu\mathrm l}
         {1\mskip-5mu\mathrm l}}}

\def \E {{\mathbb E}}
\def \L {{\mathbb L}}
\def \N {{\mathbb N}}
\def \P {{\mathbb P}}

\def \R {{\mathbb R}}
\def \Z {{\mathbb Z}}
\def \T {{\mathbb T}}

\def \lra \leftrightarrow

\def \ra {\rightarrow}
\def \ba {\begin{array}}
\def \ea {\end{array}}

\def \lra {\longrightarrow}

\def \T {{\mathbb{T}}}
\def \lra {{\leftrightarrow}}
\def \A {\mathscr{A}}

\def \un {\underline}
\def \ov {\overline}
\def \td {\tilde}
\def \Ll {\left}
\def \Rr {\right}
\def \subset {\subseteq}
\def \supset {\supseteq}

\def \mcl {\mathcal}

\def\one{\rlap{\mbox{\small\rm 1}}\kern.15em 1}

\begin{document}
\title{The contact process on finite homogeneous trees revisited}

\author{Michael Cranston\textsuperscript{1}, Thomas Mountford\textsuperscript{2}, Jean-Christophe Mourrat\textsuperscript{2,3}, Daniel Valesin\textsuperscript{4}}

\footnotetext[1]{University of California, Irvine}
\footnotetext[2]{\'Ecole Polytechnique F\'ed\'erale de Lausanne}
\footnotetext[3]{ENS Lyon, CNRS}
\footnotetext[4]{University of British Columbia}
\date{\today}
\maketitle

\begin{abstract}
We consider the contact process with infection rate $\lambda$ on $\T_n^d$, the $d$-ary tree of height $n$. We study the extinction time $\uptau_{\T_n^d}$, that is, the random time it takes for the infection to disappear when the process is started from full occupancy. We prove two conjectures of Stacey regarding $\uptau_{\T_n^d}$. Let $\lambda_2$ denote the upper critical value for the contact process on the infinite $d$-ary tree. First, if $\lambda < \lambda_2$, then $\uptau_{\T_n^d}$ divided by the height of the tree converges in probability, as $n \to \infty$, to a positive constant. Second, if $\lambda > \lambda_2$, then $\log \E[\uptau_{\T_n^d}]$ divided by the volume of the tree converges in probability to a positive constant, and $\uptau_{\T_n^d}/\E[\uptau_{\T_n^d}]$ converges in distribution to the exponential distribution of mean 1.


\end{abstract}

{\bf\large{}}\bigskip

\section{Introduction}
Let $G = (V, E)$ be a graph of bounded degree and $\lambda > 0$. The contact process on $G$ with infection rate $\lambda$ is the Markov process with state space $\{0,1\}^V$ and infinitesimal generator $\mathcal{L}$ defined, for a real function $f$ that depends on the restriction of $\xi$ to a finite subset of $V$, by
$$\mathcal{L}f(\xi) = \sum_{\substack{x \in V}} [f(\xi^{0 \to x}) - f(\xi)] + \lambda\sum_{\substack{\{x,y\} \in E:\\\xi(y) = 1}} [f(\xi^{1 \to x}) - f(\xi)], \text{ where } \xi^{i \to x}(z) = \left\{\begin{array}{ll}i &\text{if } z = x;\\ \xi(z) &\text{if } z \neq x. \end{array} \right.$$
In the usual interpretation, vertices are individuals in a population, and individuals in state 0 and 1 are healthy and infected, respectively. The dynamics given by the above generator then means that infected individuals heal at rate 1 and transmit the infection at rate $\lambda$ to each neighbour. In this Introduction, we will briefly review some properties of the contact process, and refer the reader to \cite{lig99} for a thorough exposition.

For a subset $A \subset V$, we will write $(\xi^A_t)_{t \geq 0}$ to denote the contact process on $G$ with initial configuration $\xi^A_0 = I_A$, where $I$ is the indicator function. If $A = \{x\}$, we write $(\xi^x_t)$ instead of $\left(\xi^{\{x\}}_t\right)$. When the superscript is omitted, the initial configuration of the process will be either clear from the context or unimportant. As is usual, we will abuse notation and sometimes treat $\xi \in \{0,1\}^V$ as the set $\{x \in V: \xi(x) = 1\}$.

The configuration in which all vertices are in state 0, denoted $\varnothing$, is absorbing for the contact process. For $A\subset G$, we define the extinction time for the contact process on $G$ with initial infected set $A$ by
$$\uptau_A = \inf\{t: \xi^A_t = \varnothing\}.$$

The contact process with rate $\lambda$ on $G$ is said to die out if $\P[\uptau_{\{x\}} < \infty] = 1$, and to survive otherwise. In case it survives, it is said to survive weakly (or globally but not locally) if {$\displaystyle \P\left[\limsup_{t \to \infty} \xi^x_t(x) = 1\right]=0$}, and to survive strongly (or locally) if this probability is positive. These definitions do not depend on the choice of $x$ provided that $G$ is connected. Let 
$$\begin{aligned}&\lambda_1 = \lambda_1(G) = \sup\{\lambda: \text{the contact process with parameter $\lambda$ on $G$ dies out}\};\\
&\lambda_2 = \lambda_2(G) = \inf\{\lambda: \text{the contact process with parameter $\lambda$ on $G$ survives locally}\}.
\end{aligned}$$
It is known that for $G = \Z^d$, the $d$ dimensional integer lattice, $0 < \lambda_1 = \lambda_2 < \infty$ and the process dies out at the critical point. For $G = \T^d$, the infinite tree in which all vertices have degree $d + 1$ (where $d \ge 2$), $0 < \lambda_1 < \lambda_2 < \infty$ and the process dies out at $\lambda_1$ and survives globally, but not locally, at $\lambda_2$.

An important and interesting question about the contact process is as follows. Suppose $G$ is an infinite graph, $(G_n)$ is a sequence of finite graphs with $G_n \nearrow G$, $(\xi)$ is the contact process on $G$ and $(\xi^n)$ is the contact process on $G_n$. If $n$ is large, does the behaviour of $(\xi^n)$ in any way resemble the behaviour of $(\xi)$? In particular, if $(\xi)$ has a phase transition with respect to the parameter $\lambda$, can this phase transition be identified in $(\xi^n)$ as well? Of course, since the contact process dies out on any finite graph for any value of $\lambda$, 
 the different regimes corresponding to different values of $\lambda$ cannot be defined for $(\xi^n)$ the same way they are defined for $(\xi)$. With this in mind, one considers the extinction time $\uptau_{G_n}$ and tries to determine if, asymptotically as $n \to \infty$, the law of $\uptau_{G_n}$ has different aspects depending on the value of $\lambda$.

For the case in which $G = \Z^d$ and $G_n$ is a box of $\Z^d$ with side length $n$, this has been thoroughly studied. Put briefly, it is known that if $\lambda < \lambda_1(\Z^d)$, then $\uptau_{G_n}$ grows logarithmically with the volume of $G_n$, whereas if $\lambda > \lambda_1(\Z^d)$, then $\uptau_{G_n}$ grows exponentially with the volume of $G_n$. These results are contained in \cite{eulalia}, \cite{schonmeta}, \cite{durliu}, \cite{chencp}, \cite{dursc}, \cite{tommeta}, \cite{tomexp}; see Section I.3 of \cite{lig99} for an overview.

This paper is concerned with the corresponding study for $G = \T^d$, which was initiated by Stacey in \cite{St}. Stacey's choice of finite subgraph $G_n$ was the tree $\T^d_n$, the rooted tree of height~$n$ in which all non-leaf vertices have $d$ descendants (more precisely: $\T_n^d$ has a distinguished vertex, called the root and denoted by $o$, with degree $d$; all vertices at graph distance between 1 and $n-1$ from $o$ have degree $d+1$, and vertices at distance $n$ from $o$ have degree~1). Stacey proved 
\begin{theorem}\label{thm:stsub} \cite{St} If $\lambda < \lambda_2$, then there exist $c,C \in (0,\infty)$ such that
$$
{\lim_{n \to \infty}\P\Ll[c n < \uptau_{\T_n^d}<Cn \Rr]} = 1.
$$
\end{theorem}
In other words, with high probability as $n \to \infty$, the extinction time is between fixed multiples of the logarithm of the volume of the tree. Stacey also conjectured that this result could be improved to the statement that, if $\lambda < \lambda_2$, then $\uptau_{\T_n^d}/n$ converges in probability to a constant. In this paper we confirm this conjecture:
\begin{theorem}
\label{thm:mainL}
If $\lambda < \lambda_2$, then there exists $c \in (0,\infty)$ such that 
$$
\frac{\uptau_{\T_n^d}}{n}\xrightarrow[n \to \infty]{\text{(prob.)}} c.
$$
\end{theorem}
In fact, we show that for any value of $\lambda$, $\uptau_{\T^d_n}/n$ converges in probability to some $c \in [0, \infty]$. Our proof of this fact is self-contained and quite short. Together with Theorem \ref{thm:stsub} (whose proof is very short), it implies Theorem \ref{thm:mainL}.

It should be noted that Theorem \ref{thm:mainL} does not allow us to distinguish between the two regimes delimited by the critical value $\lambda_1$. 

Stacey also studied the case $\lambda > \lambda_2$. Relying on earlier results by Salzano and Schonmann \cite{ss98}, he proved:
\begin{theorem} \cite{St}
\label{thm:stsup} If $\lambda > \lambda_2$, then for any $\beta < 1$, ${\displaystyle \lim_{n \to \infty}\P\Ll[\uptau_{\T^d_n} > e^{|\T^d_n|^\beta}\Rr]}=1.$
\end{theorem}
Here and below, $|\T^d_n|$ denotes the number of vertices, or volume, of $\T^d_n$. This means that the extinction time grows at least as fast as any stretched exponential function of the volume. Stacey conjectured that this could be improved to exponential growth, a conjecture which was partially confirmed by the following recent result.
\begin{theorem} \cite{mmvy} For each $\lambda > \lambda_1(\Z)$ and $k > 0$, there exists $c > 0$ such that the following holds. Assume $T_n$ is a sequence of trees with degree bounded by $k$ and $|T_n| \to \infty$. Let $\uptau_{T_n}$ denote extinction time for the contact process with parameter $\lambda$ on $T_n$, started from all vertices infected. Then, ${\displaystyle \lim_{n \to \infty}\P\Ll[\uptau_{T_n} > e^{c|T_n|} \Rr] = 1}.$
\end{theorem}
The reason this settled Stacey's conjecture only partially is of course the hypothesis that $\lambda > \lambda_1(\Z)$ (the critical value for the contact process on $\Z$), since we expect that $\lambda_2(\T^d) < \lambda_1(\Z)$ for every $d \ge 2$. For $d \ge 3$, this inequality is a consequence of the facts that $\lambda_1(\Z) \ge 1.539$ and that $\lambda_2(\T^d) \le (\sqrt{d}-1)^{-1}$ (see \cite[p.\ 289]{lig85} and \cite[Theorem~4.65]{lig99}). Here, we prove the following stronger version of the conjecture.
\begin{theorem}\label{thm:mainsup}
If $\lambda > \lambda_2 \ ( = \lambda_2(\T^d))$, then 
\begin{itemize}
\item[(a)] there exists $c \in (0,\infty)$ such that ${\displaystyle \lim_{n \to \infty}\frac{\log\E[\uptau_{\T_n^d}]}{|\T^d_n|}} = c$. 
\item[(b)] as $n \to \infty$, ${\uptau_{\T_n^d}}/{\E[\uptau_{\T_n^d}]}$ converges in distribution to the exponential distribution with parameter 1, that is, for any $\upalpha > 0$,
$$\lim_{n\to\infty} \P\Ll[\uptau_{\T_n^d}/\E[\uptau_{\T_n^d}] > \upalpha\Rr] = e^{-\upalpha}.$$
\end{itemize}
\end{theorem}
Although our proof of the above theorem could be made shorter by relying on some points in Stacey's proof of Theorem \ref{thm:stsup}, we have chosen to give a more self-contained proof that only relies on the estimates of \cite{ss98}.

Concerning the process started from different initial configurations, we show
\begin{theorem}
\label{thm:mainsupp}If $\lambda > \lambda_2$, then there exists $\delta > 0$ such that, for any $\upalpha > 0$ and any $n$ large enough (depending on $\upalpha$), the contact process on $\T_n^d$ satisfies
$$\inf_{A \subset \T_n^d,\; A \neq \varnothing}\;\P\Ll[\uptau_A/\E[\uptau_{\T_n^d}]  > \upalpha \Rr] > \delta e^{-\upalpha}.$$
\end{theorem}

\subsection*{Organization of the paper} The rest of the paper is organized as follows. The next section recalls the classical graphical construction and duality properties of the contact process, and fixes the notation used throughout. Section~\ref{sec:lowreg} is devoted to the proof of Theorem~\ref{thm:mainL}. The proofs of Theorems~\ref{thm:mainsup} and \ref{thm:mainsupp} are contained in Section~\ref{sec:super}. An Appendix collects some useful estimates on random walks.

The proof of Theorem~\ref{thm:mainsup} is inspired by the method developed in Section 4 of \cite{mmvy}. In that paper, we couple the contact process with independent copies of a process called the \emph{Phoenix contact process}. This is simply a process that behaves as a normal contact process until extinction, but then has the ability to recover activity after some lag. Although this method remains a useful guide for our intuition, we propose here an important simplification for its implementation, that ultimately bypasses the introduction of Phoenix contact processes and is much shorter than the proof in \cite{mmvy}.

\section{Graphical construction and duality}
Let us briefly describe the graphical construction of the contact process. Fix $\lambda > 0$ and the graph $G = (V, E)$. For each $x \in V$, let $D^x$ be a Poisson point process with parameter 1 on $[0,\infty)$ and, for each ordered pair $(x,y)$ such that $\{x,y\} \in E$, let $D^{(x,y)}$ be a Poisson point process with parameter $\lambda$ on $[0, \infty)$; these processes are taken to be independent. As a collection they are denoted by $H$ and called the graphical construction or Harris system. Points in the processes $D^x$ are called recovery marks, and points in the processes $D^{(x,y)}$ are called transmission arrows, or simply transmissions. Given a realization of $H$, $x, y \in V$ and $0 \le s \le t$, we say that $(x,s)$ is connected to $(y, t)$ by an infection path in $H$, and write $(x,s) \;\lra\; (y,t)$, if there exists a function $\gamma: [s, t] \to V$ that is right-continuous and satisfies:
$$\begin{array}{ll}\gamma(s) = x,\; \gamma(t) = y\; \text{ and, for all } r \in [s,t], &\bullet\; r \notin D^{\gamma(r)};\\&\bullet\;\gamma(r-) \neq \gamma(r) \text{ implies } r \in D^{(\gamma(r-), \gamma(r))}. \end{array}$$
Such a function is called an infection path. In words, an infection path is a path in $V$ that does not touch recovery marks and only jumps by traversing arrows.

Let $A,B,C \subset V$. We write $(x,s) \;\lra \; B \times \{t\}$ if $(x,s) \; \lra \; (y,t)$ for some $y \in B$, and similarly we write $A \times  \{s\} \; \lra \; (y,t)$ and $A \times \{s\} \;\lra\; B \times \{t\}$. We write $(x,s) \; \lra \; (y,t)$ \textit{inside $C$} if the infection path satisfies the additional requirement that $\gamma(r) \in C$ for all $r \in [s, t]$. Similarly we write $(x,s) \;\lra\; B \times \{t\}$ inside $C$, $A \times  \{s\} \; \lra \; (y,t)$ inside $C$ and $A \times \{s\} \;\lra\; B \times \{t\}$ inside $C$.

We set, for any $x \in V$ and $t \ge 0$,
\begin{equation}\xi^x_t = \{y \in V: (x,0)\;\lra\; (y,t)\}.\label{eq:graphical}\end{equation} Then, $(\xi^x_t)_{t \ge 0}$ is a version of the contact process, that is, it has the same distribution as that of the process whose generator we have given in the beginning of the Introduction, and initial configuration $I_{\{x\}}$. By setting $\xi^A_t = \cup_{x \in A}\; \xi^x_t$, we get a version of the contact process with initial configuration $I_A$.

Given $x \in V,\;t > 0$ and a realization of the graphical construction $H$, define the dual process $(\hat \xi^{(x,t)}_s)_{0 \le s \le t}$ by
$$\hat \xi^{(x,t)}_s = \{y: (y, t-s) \;\lra\; (x,t)\},$$
and for $A \subset V$, define $\hat \xi^{(A, t)}_s = \cup_{x \in A}\; \hat \xi^{(x,t)}_s$. Given $A, B \subset V$, we then have
 $$\{ \xi^A_t \cap B \neq \varnothing\} =\{\hat \xi^{B,t}_t \cap A \neq \varnothing\},$$
since both events are equal to $\{(x,0) \; \lra \; (y,t) \text{ for some } x\in A,\; y \in B\}$. This is called the duality relation for the contact process. By the time reversibility of the Poisson process, it is easy to see that for any $A$, the law of $(\hat \xi^{(A,t)}_s)_{0 \leq s \leq t}$ is the same as that of $(\xi^A_s)_{0 \le s \le t}$, that is, $(\hat \xi^{(A,t)}_s)_{0 \leq s \leq t}$ is a contact process started from $I_A$ and ran up to time $t$. Due to this fact, the contact process is said to be self-dual.

\subsection*{Summary of notation}
We denote the cardinality of a set $A$ by $|A|$, and the indicator function of $A$ by $I_A$.

The graph distance is denoted by $\mathsf{dist}$, and $B(x,r) = \{y: \mathsf{dist}(x,y) \le r\}$. We will write $x \sim y$ when $\mathsf{dist}(x,y)=1$. 

A positive integer $d \geq 2$ is fixed throughout. We will thus omit the superscript $d$ of the trees $\T^d$ and $\T^d_n$ defined above and write $\T$ and $\T_n$ instead. 

Recall that we write $o$ for the root of $\T_n$. If $x \sim y \in \T_n$ and $\mathsf{dist}(o,y) = \mathsf{dist}(o,x)+1$, we say that $x$ is the parent of $y$, $y$ is a child of $x$ and we write $x = \mathsf{p}(y) = \mathsf{p}_1(y)$. For $i \geq 1$, if $\mathsf{p}_i(y) \neq o$, we define $\mathsf{p}_{i+1}(y) = \mathsf{p}(\mathsf{p}_i(y))$. If $x = \mathsf{p}_i(y)$ for some $i$, we say that $x$ is an ancestor of $y$ and $y$ is a descendant of $x$.

For $x \in \T_n$, $\T_n(x)$ denotes the subtree of $\T_n$ which includes $x$ and its descendants. We also write $\T_n(x, k) = \T_n(x) \cap B(x, k)$. For $0 \leq m \leq n$, we write $\L_n(m) = \{x \in \T_n: \mathsf{dist}(o,x) = m\}$. For $x \in \T_n$, $0 \le m \le n$, we write $\L_n(x,m) = \{y \in \T_n(x):\mathsf{dist}(x,y) = m\}.$ 

For $a, b\in\mathbb{R}$ with $a<b$, we denote by $D_n[a, b]$ the set of functions $f: [a, b] \to \{0,1\}^{\T_n}$ that are right-continuous with left limits.

We will use the notation explained above for the contact process: $\xi^x_t$ is the process started from $\xi_0 = I_{\{x\}}$ and $\xi^A_t$  is the process started from $I_A$. In order not to make the notation too heavy, we do not include in this notation the graph in which the process is being considered, so this will always be clear from the context. 

\section{Subcritical and intermediate regimes}
\label{sec:lowreg}
In this section we assume that $\lambda < \lambda_2(\T)$.

Let ${\displaystyle b^* = \sup\Ll\{b: \limsup_{n \to \infty} \P[\uptau_{\T_n}> bn] > 0\Rr\}}$. By Theorem \ref{thm:stsub}, $b^* \in (0, \infty)$. 

\begin{lemma}
For any $b < b^*$, there exist infinitely many values of $N \in \mathbb{N}$ such that the contact process $\xi$ on $\T_N$ satisfies the following. There exist $0 \le i \le j \le N$ such that, for any $x \in \L_{N}(i)$, 
$$
\P\Ll[\xi^x_{bN} \cap \mathbb{L}_N(j) \neq \varnothing \Rr] > (N^3 d^i)^{-1}.
$$
\end{lemma}
\begin{proof}
Fix $b < b^*$. Using the definition of $b^*$, we see that there exists $\delta > 0$ such that, for infinitely many values of $N$, $\P\Ll[\uptau_{\T_N} > bN\Rr] > \delta.$ Since 
$$\P\Ll[\uptau_{\T_N} > bN\Rr] \le \sum_{0 \leq i,j \leq N} \P\Ll[\;\xi^{\L_N(i)}_{bN} \cap \L_N(j) \neq \varnothing\;\Rr],$$
there exist $i,j$ such that $\P\Ll[\;\xi^{\L_N(i)}_{bN} \cap \L_N(j) \neq \varnothing\;\Rr] > \delta/N^2$. Since, by duality,
$$\P\Ll[\;\xi^{\L_N(i)}_{bN} \cap \L_n(j) \neq \varnothing\;\Rr] = \P\Ll[\;\xi^{\L_N(j)}_{bN} \cap \L_N(i) \neq \varnothing\;\Rr],$$
we may assume that $i \le j$. We now have
$$\delta/N^2 < \P\Ll[\;\xi^{\L_N(i)}_{bN} \cap \L_N(j) \neq \varnothing\;\Rr]\leq \sum_{x \in \L_N(i)} \P\Ll[\;\xi^x_{bN} \cap \L_N(j) \neq \varnothing\;\Rr]$$
and the probability inside the sum does not depend on $x \in \L_N(i)$. With the observation that $|\L_N(i)| = d^i$ and the added requirement that $N > 1/\delta$, the inequality of the lemma then holds.
\end{proof}

\begin{proofof}\textit{Theorem \ref{thm:mainL}}. Fix $\epsilon > 0$. We will show that, for $n$ large enough, $$\P\Ll[\uptau_{\T_n} > (b^* - \epsilon)n\Rr] > 1 - \epsilon.$$
Together with the definition of $b^*$, this will imply that $\frac{\uptau_{\T_n}}{n}$ converges to $b^*$ in probability.

We let $b = b^* - \epsilon/2$ and choose $N, i, j$ corresponding to $b$ in the above lemma. We assume $N$ is large enough that
\begin{equation}
\label{eq:NlogN}\frac{N}{\frac{9 \log N}{\log d} + N} \;(b^*-\epsilon/2)> b^* - \epsilon.
\end{equation}
Now fix $n$ much larger than $N$ and define 
$$M = \Bigl\lfloor \frac{8 \log N}{\log d} \Bigr\rfloor,\quad k = \Bigl\lfloor \frac{n}{M + N}\Bigr\rfloor-1,\quad n_1 = k(M+i).$$
Note that
\begin{equation}\nonumber
n_1 + (k-1)(j-i) + N \le kM +(k-1)j + i + N \leq kM + (k-1)N + N + N\leq (k+1)(M+ N) \leq n,\end{equation}
so
\begin{equation}\label{eq:helpn_1} \text{if } 0 \leq h \leq n_1 + (k-1)(j-i) \text{ and } y \in \L_n(h), \text{ then } \T_n(y, N) \text{ is a tree of height }N. \end{equation}

For each $x \in \L_n(n_1)$, we will define a nested sequence of events
$$A(x,k) \subset A(x,k-1) \subset \cdots \subset A(x,1)$$
such that
$$\bigcup_{x\in \L_n(n_1)} A(x,k) \subset \{\uptau_{\T_n} > bkN\}$$
and
$$\lim_{n \to \infty} \P \left[\bigcup_{x\in \L_n(n_1)} A(x,k)  \right] = 1.$$

Starting with an $x\in\L_n(n_1)$, let $y_0(x) = x$ and let $z_0(x)$ be an arbitrary vertex in $\L_n(x, i)$. Define the event
$$A(x,1) = \{\;(z_0(x),0) \;\lra \; \L_n(x,j) \times \{bN\} \text{ inside } \T_n(x, N)\;\}.$$
On $A(x,1)$, define $z_1(x)$ as a vertex of $\L_n(x,j)$ such that $(z_0(x), 0) \; \lra \; (z_1(x), bN)$ inside $\T_n(x,N)$. Also let $y_1(x) = \mathsf{p}_i(z_1(x))$; note that $y_1(x) \in \L_n(y_0(x),j-i)$. Then define the event $A(x, 2) \subset A(x,1)$ by
$$A(x,2) = \{\;(z_1(x), bN) \; \lra \; \L_n(y_1(x), j) \times \{2bN\} \text{ inside } \T_n(y_1(x), N)\;\}.$$
On this event, let $z_2(x)$ be a vertex of $\L_n(y_1(x),j)$ such that $(z_1(x), bN) \; \lra \;(z_2(x), 2bN)$ inside $\T_n(y_1(x),N)$. Then let $y_2(x) = \mathsf{p}_i(z_2(x))$. Note that $y_2(x) \in \L_n(y_1(x), j-i)$. Using (\ref{eq:helpn_1}), we then repeat this definition until we have obtained $A(x,k),\;z_k(x)$. Figure 1 may clarify these definitions.

\begin{figure}[htb]
\begin{center}
\setlength\fboxsep{0pt}
\setlength\fboxrule{0.5pt}
\fbox{\includegraphics[width = 1.0\textwidth]{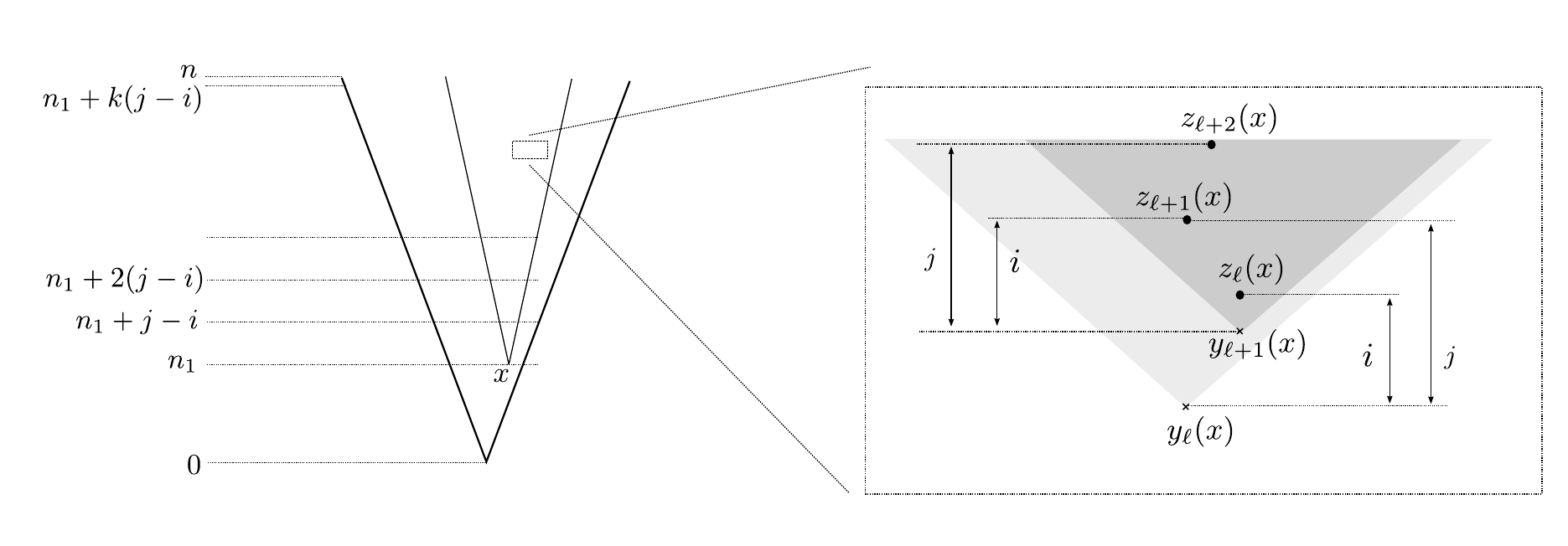}}
\end{center}
\caption{The sequences $y_1(x), \ldots, y_k(x),\; z_1(x),\ldots, z_k(x).$}
\label{figDual}
\end{figure}

We have, for any $x \in \L_n(n_1)$,
$$\P[A(x,1)] \ge (N^3 d^i)^{-1},\qquad \P[A(x, \ell +1)\;|\;A(x,\ell)] \ge (N^3 d^i)^{-1},\; 1 \le \ell < k,$$
so that $\P[ A(x,k) ] \ge (N^3 d^i)^{-k}.$
The sequences of events $(A(x,1),\ldots, A(x,k))$ are independent (in the variable $x$), since $(A(x,1),\ldots, A(x,k))$ only depends on the graphical construction inside $\T_n(x)$, and the sets $\T_n(x)$ for $x \in \L_n(n_1)$ are disjoint. Thus,
\begin{equation}\P\Ll[ \mathop{\bigcup}_{x \in \L_n(n_1)} A(x,k)\Rr] \ge \P\Ll[\;\mathrm{Bin}\Ll(|\L_n(n_1)|,\; (N^3 d^i)^{-k}\Rr) > 0\Rr].\label{eq:lBin}\end{equation}
The expectation of the above Binomial is $$d^{n_1}(N^3\cdot d^i)^{-k} = (N^{-3}\cdot d^{M})^k \ge (N^{-3} \cdot d^{\frac{4\log N}{\log d}})^k = N^k,$$
so, as $n \to \infty$ (and thus $k \to \infty$), the right-hand side of (\ref{eq:lBin}) converges to 1.

If $A(x,k)$ occurs for some $x$, then $(x,0)\; \lra\; (z_k(x),\; bkN)$, so $\uptau_{\T_n} > bkN$. Finally, note that
$$\begin{aligned}bkN &= (b^* - \epsilon/2) \left(\Bigl \lfloor \frac{n}{N + \lfloor 8\log N/\log d \rfloor} \Bigr \rfloor - 1 \right) N \\&\qquad\qquad\qquad\qquad\qquad\qquad\qquad\qquad> (b^* - \epsilon/2) \;\frac{nN}{N + (9\log N/\log d)} > (b^* - \epsilon)n,\end{aligned}$$
where the last inequality follows from (\ref{eq:NlogN}). This completes the proof.
\end{proofof}

\section{Supercritical regime}
\label{sec:super}
In this section we prove Theorems \ref{thm:mainsup} and \ref{thm:mainsupp}. We start with an outline of our approach. 

We follow a recursive scheme based on the following elementary observations.
Assume given a graphical construction $H$ for the contact process $(\xi^A_t)$ on $\T_n$, and let $m \le n$. For each $x \in \L_n(m)$, we can use the restriction of $H$ to the subtree $\T_n(x)$ to define a contact process $(\xi^{A \cap \T_n(x)}_{\T_n(x),t})_{t \ge 0}$ on this subtree by setting
$$\xi^{A \cap \T_n(x)}_{\T_n(x),t}(y) = I\{A \cap \T_n(x) \times \{0\} \;\leftrightarrow\; (y,t) \text{ inside } \T_n(x)\}, \qquad y \in \T_n(x),\;t\ge 0.$$
The processes $\{(\xi^{A \cap \T_n(x)}_{\T_n(x),t}): x\in \L_n(m)\}$ are evidently all defined in the same probability space. Moreover, they are independent and satisfy, for any $x \in \L_n(m)$, $y \in \T_n(x)$ and $t \ge 0$, $\xi^{A}_t(y) \ge \xi^{A \cap \T_n(x)}_{\T_n(x),t}(y)$.

Our proof is divided into levels, which are numbered from 1 to 4. In each level $k \in \{1,2,3, 4\}$, we obtain a lower bound on the probability of some good event involving the contact process on $\T_n$ (for any large enough $n$) within some time scale $t^{(k)}_n$. The treatment of each level after the first appeals to the previous level, according to the following scheme. In level $k \ge 2$, we decompose the height $n$ of $\T_n$, writing $n = M^{(k)}_{n} + N^{(k)}_{n}$ (this notation will only be used in this outline). We apply the result of level $k-1$ to the $d^{M^{(k)}_{n}}$ contact processes $\{(\xi_{\T_n(x), t})_{t \ge 0}: x \in \L_n(M^{(k)}_{n})\}$. Since the subtrees in which these processes occur have height $N^{(k)}_n$, the time scale $t^{(k)}_n$ is chosen larger than $t^{(k-1)}_{N^{(k)}_{n}}$, so that the processes can satisfy the pertinent event of level $k-1$. We then argue that the good event of level $k$ follows from the occurrence of sufficiently many good events of level $k-1$, which in turn has high probability.

It should be mentioned that Stacey's proof of Theorem \ref{thm:stsup} also follows a similar recursive strategy. We believe the key point that allowed us to improve his result is the use we make of a certain coupling result (Corollary \ref{prop:coupling}, obtained in level 3) in level 4.

Before starting on level 1, we state a general result about the contact process that will be quite useful. Let $G = (V, E)$ be a locally finite graph and assume given a graphical construction for the contact process with rate $\lambda > 0$ on $G$. Given $A, \A \subset V$ and $0 < t_0 < t$, define
\begin{equation}
\mathcal{N}^A_{\A,t_0}(t) = \max\left\{\begin{array}{c}k: \text{ there exist } 0 \le s_1 < s_2 < \cdots < s_k < t - t_0\\\text{such that } s_{i+1} - s_i \ge t_0 \text{ and } \xi^A_{s_i} \cap \A \neq \varnothing \text{ for all } i \end{array} \right\}.\label{eq:defN}\end{equation}
In words, $\mathcal{N}^A_{\A,t_0}(t)$ is the maximal number of disjoint subintervals of length $t_0$ that we can extract from $[0,t]$ with the restriction that at the starting point of each subinterval, at least one vertex of $\A$ must be infected by $\xi^A$.

\begin{lemma}
\label{lem:Rep}  Assume that, for numbers $t_0$ and $\upepsilon_0$ and sets $\A \subset V$ and $E\subset \{0,1\}^V$,  we have
\begin{equation}\label{eq:cond1Rep}\P\Ll[\exists t \le t_0: \xi^B_{t} \in E\Rr] > \upepsilon_0 \qquad \text{for all } B \text{ with } B \cap \A \neq \varnothing.\end{equation}
Let $\upkappa^A = \inf\{t: \xi^A_t \in E\}$. Then, for any $N > 0,\;t > t_0$ and $A \in \{0,1\}^V$ we have
\begin{equation}\label{eq:repCon}\P\Ll[\upkappa^{A} > t,\; \mathcal{N}^A_{\A,t_0}(t) \ge N\Rr] \leq \Ll(1-\upepsilon_0\Rr)^{N}.\end{equation}
\end{lemma}
Since this is a simple consequence of the Markov property, we omit the proof. The idea is that in each of the intervals of length $t_0$ that appear in the definition of $\mathcal{N}^A_{\A,t_0}(t)$, the process has probability $\upepsilon_0$ of reaching $E$, so the probability of making $N$ attempts and failing at them all is less than $(1-\upepsilon_0)^N$.

In the rest of this section, we always assume that $\lambda > \lambda_2(\T)$.

\subsection{Level 1: the Salzano - Schonmann estimates}
For our starting level, we simply gather some estimates of \cite{ss98}; the most important of them is Lemma \ref{lem:ss1}(i) below. It implies the extinction time $\uptau_{\T_n}$ is at least linear in $n$ with non-vanishing probability, but more importantly, that there are some deterministic times in which the root of the tree has non-vanishing probability of being infected.

\label{ss:l1}
\begin{lemma} \cite{ss98} There exist $\sigma > 0$, $K, S > 0$ such that, for $n$ large enough and the contact process on $\T_n$,\medskip\\
$(i.)\; \P\Ll[\xi^{o}_{iS}(o) = 1\Rr] > \sigma$ for $i = 0, 1,\ldots, \lfloor n/K \rfloor$;\medskip\\
$(ii.)\; \P\Ll[\big|\xi^{o}_{Sn/(2K)} \cap \mathbb{L}_{n}(n)\big| >(\frac{2}{3} d)^n \Rr] > \sigma.$
\label{lem:ss1} 
\end{lemma}

\begin{lemma} \cite{ss98}
\label{lem:ss2} For any $\theta < 1$ there exists $\bar c, \ell > 0$ such that, for any $n$ and any $x, y \in \T_n$, the contact process on $\T_n$ satisfies
$$\P\Ll[\xi^{x}_{\ell\cdot \mathsf{dist}(x,y)}(y) = 1 \Rr] > \bar c \cdot \theta^{\mathsf{dist}(x,y)}.$$
\end{lemma}

Putting these two lemmas together, we get
\begin{corollary}\label{cor:2lem}
Assume that $n_1 \le n$ and $A \subset \T_n$ is such that $A \cap B(o,n_1)\neq \varnothing$. Then,
$$\P\Ll[ |\xi^A_t \cap \L_n(n_1)| \ge (2d/3)^{n_1} \text{ for some } t \le \left(\ell + S/(2K)\right)n_1\Rr] \ge \bar c \theta^{n_1}\cdot \sigma.$$
\end{corollary}
Indeed, given $x \in A$ with $\mathsf{dist}(o,x) \leq n_1$, with probability larger than $\bar{c}\theta^{\mathsf{dist}(o,x)}\geq \bar{c}\theta^{n_1}$, we have $(x,0)\;\lra\;(o,s)$ for some $s \leq \ell \cdot n_1$. Conditioned on this event, with probability larger than $\sigma$ the resulting infection present at $o$ at time $s$ further propagates, reaching $(2d/3)^{n_1}$ vertices of $\mathbb{L}_n(n_1)$ at time $s + n_1 \cdot S/(2K)$.\\

\noindent \textbf{Summary of constants.
} From now on, we denote $\bar d = \frac{2}{3}d$. Once and for all, we fix $\theta < 1$, together with constants $v_0 < v_1$ chosen so that 
\begin{equation}
\label{eq:defu}
1 < 1/\theta < v_i^{1/6} < v_i^{1/2} < \bar d, \quad i = 0, 1.
\end{equation}
For this value of $\theta$, we fix the constants $\ov{c}$ and $\ell$ as given by Lemma~\ref{lem:ss2}. The constants $\sigma$, $K$ and $S$ from Lemma~\ref{lem:ss1} will also be kept fixed throughout.

\subsection{Level 2: a set of configurations with high return probability}
\label{ss:l2}
For $n \in \N$ large enough, we can choose $n_1 \in \N$ and $u_n \in [v_0,\;v_1]$ such that 
\begin{equation}\label{eq:decomp} n = n_1 + n_2,\quad \text{where } n_2 = (u_n)^{n_1}.\end{equation}

In this section we perform our first recursion; let us give a rough sketch of what this will be. Using Lemma \ref{lem:ss1}(i.), we will argue that, if many of the roots of the subtrees $\{\T_n(x):x \in \L_n(n_1)\}$ are infected at time $t$, then for an amount of time that is linear in $n_2$ (hence of order $(u_n)^{n_1}$), in every $S$ time units some of these roots will again be infected. Every time one of them is infected, the infection gets an attempt of travelling down to the root of $\T_n$, and from there propagating back up to many other subtrees rooted in $\L_n(n_1)$. By Corollary~\ref{cor:2lem}, the probability that an attempt is successful is $\bar c \theta^{n_1}\cdot \sigma$. Comparing $\theta$ to $u_n$, it will be easy to see that with high probability we will have a successful attempt.

Using these ideas, we will obtain a set $\mathcal{G}_n \subset \{0,1\}^{\T_n}$ (not containing the empty configuration) with the property that, if $\xi_t \in \mathcal{G}_n$, then with probability larger than $1-e^{-n^{1/3}}$, $\xi_{t+\sqrt{n}}$ is also in $\mathcal{G}_n$.

\begin{proposition}\label{prop:main}
For $n$ large enough, there exists $\mathscr{G}_n \subset \{0,1\}^{\T_n}$ such that, if $A \in \mathscr{G}_n$ and $t \in [n^{1/2}/{ 4},\; n^{1/2}]$,
\begin{equation}
\label{eq:prop1}\P\Ll[\xi^{A}_t \in \mathscr{G}_n \text{ and } \exists t' \leq n^{1/2}/8: \xi^A_{t'}(o) = 1 \Rr] > 1-e^{-n^{1/3}}.\end{equation}
Additionally, for all $t \in [n^{1/2}/{2},\; n^{1/2}]$,
\begin{equation}
\label{eq:prop3} \P\Ll[\xi^{o}_t \in \mathscr{G}_n\Rr] > \sigma(1 - e^{n^{-1/3}}) > \sigma/2.
\end{equation}
Finally, if $A, A' \in \mathscr{G}_n$ and $t \in [n^{1/2}/{ 2},\; n^{1/2}]$,
\begin{equation}
\label{eq:prop4} \P\Ll[\xi^{A}_t \cap A' \neq \varnothing\Rr] > n^{-1}.
\end{equation}
\end{proposition}
\begin{proof}
We fix $n$ large enough that the decomposition (\ref{eq:decomp}) is possible; we also assume that $n$ is large enough that the trees of height $n_1$ and $n_2$ satisfy the properties stated in Lemma \ref{lem:ss1}. 

Define
$$F =\left\{\begin{array}{l}f \in D_n([a,b]): 0\leq a < b; \text{ for every interval } \text{$J \subset [a,\;b]$}\\\text{with $|J| = S$, there exists $s \in J$ such that } f(s)\cap \mathbb{L}_{n}(n_1) \neq \varnothing\end{array}\right\}.$$
Let us show that
\begin{equation}\text{if }|A \cap \mathbb{L}_{n}(n_1)| > \bar d^{n_1}, \text{then } \P\Ll[\;(\xi^{A}_t)_{0 \leq t \leq 2n^{1/2}} \in F\Rr] > 1 - e^{-n^{1/2}}. \label{eq:goodstart}\end{equation}
Indeed, if $|A \cap \mathbb{L}_{n}(n_1)| > \bar d^{n_1}$, then, by applying Lemma \ref{lem:ss1} to the trees $\T_n(x)$ for $x \in A \cap \L_n(n_1)$, we see that for $i = 0, 1, \ldots, \lfloor n_2/K \rfloor$, $|\xi^{A}_{iS} \cap \mathbb{L}_{n}(n_1)|$ stochastically dominates a $\mathrm{Binomial}(\bar d^{n_1}, \sigma)$ random variable. The probability that $\xi^{A}_{iS} \cap \mathbb{L}_{n}(n_1)$ is empty for some $i \in \{0, 1, \ldots, \lfloor n_2/K \rfloor\}$ is thus smaller than
$$\frac{n_2}{K}\cdot (1-\sigma)^{\bar d^{n_1}} \leq \frac{n_2}{K}\cdot e^{-\sigma \cdot \bar d^{n_1} } \leq (2 u_n)^{n_1} \cdot e^{-\sigma \cdot \bar d^{n_1}} = e^{\log( 2 u_n)n_1 - \sigma \cdot \bar d^{n_1}}< e^{-(\sigma/2) \cdot \bar d^{n_1}}< e^{-n^{1/2}}$$
if $n_1$ is large, since $n^{1/2} = (n_1 + (u_n)^{n_1})^{1/2} <(2(u_n)^{n_1})^{1/2} = \sqrt{2} ((u_n)^{1/2})^{n_1} < \bar d^{n_1}$, as $(u_n)^{1/2} < \bar d$. So, outside of probability $e^{-n^{1/2}}$, the desired property of $\mathbb{L}_n(n_1)$ never being empty for more than $S$ time units is satisfied up to time $\lfloor n_2/k \rfloor > 2n^{1/2}$. This proves (\ref{eq:goodstart}).

For $A \subset {\T_n}$, let $$\upphi(A) = \P\Ll[\;(\xi^{A}_t)_{0 \leq t \leq n^{1/2}} \in F\Rr] \qquad \text{and} \qquad \mathscr{G}_n = \left\{A \subset \T_n: \upphi(A) > 1 - e^{-\frac{1}{2}n^{1/2}}\right\}.$$
Note that (\ref{eq:goodstart}) implies that 
\begin{equation}\label{eq:goodstart2}\{A: |A \cap \mathbb{L}_{n}(n_1)| > \bar d^{n_1}\} \subset \mathscr{G}_n.\end{equation} 

We are now ready to start our proof of (\ref{eq:prop1}). Fix $A \in \mathscr{G}_n$ and $t \in [n^{1/2}/4,\;n^{1/2}]$. Define $$\upkappa = \inf\{s \geq 0: |\xi^{A}_s \cap \mathbb{L}_{n}(n_1)|>\bar d^{n_1} \}.$$ We have
\begin{equation}\begin{split}\P\Ll[\xi^{A}_{t} \notin \mathscr{G}_n \Rr] \leq  \P\Ll[(\xi^{A}_s)_{0\leq s \leq t} \notin F\Rr] &+ \P\Ll[(\xi^{A}_s)_{0\leq s \leq t} \in F,\;\upkappa > t\Rr] + \P\Ll[\xi^{A}_{t} \notin \mathscr{G}_n,\; \upkappa \leq t \Rr].\label{eq:upkappa}\end{split}\end{equation}
The first term on the right-hand side is less than $e^{-\frac{1}{2}n^{1/2}}$ by the definition of $\mathscr{G}_n$. Let us bound the second and third terms, starting with the third term. Note that
$$\begin{aligned}\E\left[\left(1-\upphi(\xi^{A}_{t})\right) \cdot I_{\{\upkappa \leq t\}} \right]&=\sum_{B \subset \T_n} \mathbb{P}\left[\upkappa \leq t,\;\xi^A_t = B\right] \cdot \mathbb{P}\left[(\xi^B_s)_{0\leq s \leq n^{1/2}} \notin F\right]\\
&=\P\Ll[\upkappa \leq t,\;(\xi^{A}_s)_{t \leq s \leq t+n^{1/2}} \notin F  \Rr] <e^{-n^{1/2}}
\end{aligned}$$
by (\ref{eq:goodstart}). Then, by Markov's inequality,
$$\P\Ll[\xi^{A}_{t} \notin \mathscr{G}_n,\; \upkappa \leq t  \Rr] = \P\Ll[\left(1-\upphi(\xi^{A}_{t})\right)\cdot I_{\{\upkappa \leq t \}} \geq e^{-\frac{1}{2}n^{1/2}}\;\Rr] \leq \frac{e^{-n^{1/2}}}{e^{-\frac{1}{2}n^{1/2}}} = e^{-\frac{1}{2}n^{1/2}}.$$

We will now bound $\P\Ll[(\xi^{A}_s)_{0\leq s \leq t} \in F,\;\upkappa > t\Rr]$. Let $r_0 = \left(\ell + \frac{S}{2K}\right)n_1$. We first note that, since $S \ll r_0$ when $n_1$ is large, on $\left\{(\xi^A_s)_{0 \le s\le t} \in F\right\}$ we can find $t_1 < t_2 < \cdots < t_{\lfloor t/(2r_0) \rfloor}$ such that $t-t_{\lfloor t/(2r_0) \rfloor} > r_0$, $t_{i+1} - t_{i}>r_0$ and $\xi^A_{t_i} \cap \L_n(n_1) \neq \varnothing$ for each $i$. This implies that
$$\left\{(\xi^A_s)_{0\le s \le t} \in F\right\} \subset \left\{\mathcal{N}^A_{\L_n(n_1), r_0}(t) \ge \lfloor t/(2r_0) \rfloor\right\},$$
where $\mathcal{N}^A_{\L_n(n_1), r_0}(t)$ is as in (\ref{eq:defN}).
Therefore,
\begin{equation}\P\Ll[(\xi^{A}_s)_{0\leq s \leq t} \in F,\;\upkappa > t\Rr] \le \P\Ll[\mathcal{N}^A_{\L_n(n_1), r_0}(t) \ge \lfloor t/(2r_0)\rfloor,\;\upkappa > t\Rr].\label{eq:FtoN}\end{equation}
We now bound the right-hand side with Lemma \ref{lem:Rep}. We set the parameters of Lemma \ref{lem:Rep} as follows:
$$\A = \mathbb{L}_{n}(n_1),\quad t_0 = r_0,\quad E = \{B \subset \T_n: |B \cap \mathbb{L}_{n}(n_1)| >\bar d^{n_1}\},\quad \upepsilon_0 = \bar c \sigma \theta^{n_1}.$$
Note that (\ref{eq:cond1Rep}) follows from Corollary \ref{cor:2lem}. The right-hand side of (\ref{eq:FtoN}) is then less than
$$\begin{aligned}&\exp\left(-\bar c \sigma \theta^{n_1}\cdot \left\lfloor \frac{t}{2r_0} \right\rfloor \right) \le \exp\left(-\frac{\bar c\sigma}{16\left(\ell + \frac{S}{2K} \right)} \cdot \frac{1}{n_1}\cdot \theta^{n_1}\cdot (n_1 + u_n^{n_1})^{1/2} \right) < \frac{1}{4}e^{-n^{1/3}}\end{aligned}$$
if $n$ is large enough, since $\theta u_n^{1/2} > u_n^{1/3}$ by (\ref{eq:defu}).

Going back to (\ref{eq:upkappa}), we have proved that
$$\P\Ll[ \xi^A_t \in \mathcal{G}_n \Rr] \ge 1-e^{-\frac{1}{2}n^{1/2}}-e^{-\frac{1}{2}n^{1/2}}-\frac{1}{4}e^{-n^{1/3}} > 1 - \frac{1}{2}e^{-n^{1/3}}.$$
By a simpler application of Lemma \ref{lem:Rep} than the one explained above, we also get
$$\P\Ll[ \exists t'\leq n^{1/2}/8: \xi^A_{t'}(o) = 1 \Rr] \ge 1-\frac{1}{2}e^{-n^{1/3}}.$$
This completes the proof of (\ref{eq:prop1}).

(\ref{eq:prop3}) follows from (\ref{eq:goodstart2}), (\ref{eq:prop1}), Lemma \ref{lem:ss1}(ii.) and the fact that $(S/2K)n_1 < n^{1/2}/4$.

Let us now prove (\ref{eq:prop4}). Define the events
$$\begin{aligned}
&B_1 = \left\{\exists s^* \in [t-2S-2\ell n_1,\;t-S-2\ell n_1],\; x^*\in \mathbb{L}_{n}(n_1): A \times \{0\} \leftrightarrow (x^*, s^*)\right\};\\
&B_2 = \left\{\exists s^{**} \in [t-S,\;t],\; x^{**}\in \mathbb{L}_{n}(n_1): (x^{**},s^{**}) \leftrightarrow A' \times \{t\}\right\}.
\end{aligned}$$
We have $\P[B_1] \geq \P\Ll[(\xi^{A}_s)_{0\leq s \leq t} \in F \Rr] \geq 1 - e^{-n^{1/3}}$ and, by the same consideration for the dual process, $\P[B_2]\geq 1 - e^{-n^{1/3}}$. Also,
$$\begin{aligned}
\P\Ll[\xi^{A}_t \cap A' \neq \varnothing \;|\;B_1 \cap B_2 \Rr] &{\geq}\; \P\Ll[(x^*,s^*)\;\lra\;(o, s^* + \ell n_1)\;\lra\;(o, s^{**} - \ell n_1)\;\lra\;(x^{**}, s^{**}) \;|\; B_1 \cap B_2\Rr]\\&{\geq}\; {\ov{c}^2}\theta^{2n_1}\cdot e^{-2S},
\end{aligned}$$
by Lemma \ref{lem:ss2} and the fact that $s^{**} - s^{*} \in [2\ell n_1,\; 2\ell n_1 + 2S]$. We have thus shown that
$$\P\Ll[\xi^{A}_t \cap A' \neq \varnothing \Rr] \geq (1-2e^{-n^{1/3}})\cdot \bar{c}^2 \theta^{2n_1}\cdot e^{-2S} > 1/n$$
for $n$ large, since $\theta^2 > 1/u_n$.
\end{proof}

\subsection{Level 3: survival and coupling for time $\exp\big(d^{n^{1/5}}\big)$}
\label{ss:l3}
Define
$$n = m_1 + m_2,\; \text{where } m_1 = \lfloor n^{1/4} \rfloor.$$
Given $\xi \in \{0,1\}^{\T_n}$ and $x \in \L_n(m_1)$, we will say that $\xi \cap \T_n(x) \in \mathscr{G}_{m_2}$ if $\xi \cap \T_n(x)$, when seen as a configuration for a tree of height $m_2$, is in the set $\mathscr{G}_{m_2}$ defined in the previous subsection. For $A \subset {\T_n}$, define
$$\Gamma(A) = \left|\left\{x \in \mathbb{L}_{n}(m_1): A \cap \T_n(x) \in \mathscr{G}_{m_2}\right\} \right|.$$

Suppose that we have $\Gamma(\xi_0) = a \in (0,d^{m_1}) \cap \Z$ and that $t \in [(1/2)m_2^{1/2}, m_2^{1/2}]$. Applying the results of level 2, it will be easy to prove that with high probability, for all $x \in \L_n(m_1)$ such that $\xi_0 \cap \T_n(x) \in \mathcal{G}_{m_2}$, we will also have $\xi_t \cap \T_n(x) \in \mathcal{G}_{m_2}$. In other words, with high probability $\Gamma(\xi_t) \ge \Gamma(\xi_0)$. Additionally, from time 0 to time $t$, the infection gets an attempt to reach a subtree $\bar{T} \in \{\T_n(x):x\in\L_n(m_1),\xi_0 \cap \T_n(x)\notin \mathcal{G}_{m_2}\}$ and spread sufficiently inside it that we get $\xi_t \cap \bar{T} \in \mathcal{G}_{m_2}$. If such an attempt is successful, we have $\Gamma(\xi_t) > \Gamma(\xi_0)$.

With this in mind, we will argue that, if $0 = t_0 < t_1 < \ldots$ is a sequence of times with $\frac{1}{2}m_2^{1/2} \leq t_{i+1} - t_i \leq m_2^{1/2}$ for each $i$ and $(\xi_t)$ is the contact process on $\T_n$ with an arbitrary initial configuration, then $(\Gamma(\xi_{t_0}),\;\Gamma(\xi_{t_1}),\ldots)$ is stochastically larger than a certain Markov chain $Z^{(n)}$ with state space $(-\infty,\;d^{m_1}] \cap \Z$ that tends to move much more to the right than to the left.

It will be convenient to write
$$h_k(n) = d^{\lfloor n^{1/k} \rfloor},\; \text{ for } k \in \N.$$
Note that $|\L_n(m_1)| = d^{m_1} = d^{\lfloor n^{1/4} \rfloor} = h_4(n)$.

Let us now define the transition kernel $P(\cdot, \cdot)$ of $Z^{(n)}$. Let $p^{(n)}$ be the probability mass function for the Binomial$\left(h_4(n),\; \exp(-n^{1/3})\right)$ distribution. Define
\begin{equation}
\label{def:Z}
\begin{array}{l} \underline{\text{for } 1 \leq k < h_4(n):}\medskip\\ P(k, k+1) = p^{(n)}(0)\cdot \frac{\bar c\sigma}{2}\theta^{2m_1};\medskip\\P(k,k) = p^{(n)}(0)\cdot \left(1-\frac{\bar c\sigma}{2}\theta^{2m_1} \right);\medskip\\P(k, k-a) = p^{(n)}(a),\; a > 0;\end{array} \qquad \begin{array}{l}\underline{\text{for }k = h_4(n):}\medskip\\ P(k, k-a) = p^{(n)}(a),\; a \ge 0;\medskip\\\;\medskip\\\;\end{array} \qquad \begin{array}{l}\underline{\text{for }k \leq 0:} \medskip\\P(k,k)=1\medskip\\\;\medskip\\\;\end{array}
\end{equation}
(obviously, $P(k, \ell) = 0$ for all values of $\ell$ for which we have not explicitly defined it).

\begin{lemma} \label{lem:compRW}
For $n$ large enough and the contact process $(\xi_t)$ on $\T_n$, the following holds. If $0 = t_0 < t_1 < \ldots$ are such that $\frac{1}{2}m_2^{1/2} \leq t_{i+1} - t_i \leq m_2^{1/2}$ for each $i$ and $\Gamma(\xi_0) = a$, then $\left(\Gamma(\xi_{t_i})\right)_{i \geq 0}$ stochastically dominates the Markov chain $(Z^{(n)}_i)_{i \geq 0}$ with initial state $Z^{(n)}_0 = a$.
\end{lemma}
\begin{proof}
Fix an initial configuration $\xi_0$ and ${t} \in [m_2^{1/2}/2,\;m_2^{1/2}]$. 
Let $a = \Gamma(\xi_0)$; we will for now assume that $1\leq a < h_4(n)$. Define, for $x \in \L_{n}(m_1)$ and $s \ge 0$,
$$\eta[x]_s = \left\{y \in \T_n(x): \xi_0 \cap \T_n(x)\times\{0\} \;\lra\;(y,s) \text{ inside }\T_n(x) \right\}.$$
Obviously, $\eta[x]_s$ is simply a contact process on $\T_n(x)$ with initial configuration $\xi_0 \cap \T_n(x)$. As usual, we will abuse notation and treat $\eta[x]_s$ as an element of $\{0,1\}^{\T_n(x)}$. 

Let
$$V = \left|\left\{x \in \L_{n}(m_1): \xi_0 \cap \T_n(x) \in \mathscr{G}_{m_2};\;(\eta[x]_s)(x) = 1 \text{ for some }s \leq m_2^{1/2}/8;\;\eta[x]_t \in \mathscr{G}_{m_2} \right\}\right|.$$
Using (\ref{eq:prop1}) and the fact that the processes $\{(\eta[x]_s)_{0 \leq s \leq t}: x\in \L_{n}(m_1)\}$ are independent, we see that $V$ is stochastically larger than a Binomial($a,\;1-\exp(-n^{1/3})$) random variable. This implies that $V$ is stochastically larger than $a - X$, where $X$ is a random variable with Binomial($h_4(n),e^{-n^{1/3}}$) distribution.

On the event $\{V = a\}$, we can choose $x^* \in \L_{n}(m_1)$, $s^* \leq m_2^{1/2}/8$ such that $\left(\eta[x^*]_{s^*}\right)(x^*) = 1$, and also choose $y^*$ with $\xi_0(\T_n(y^*)) \notin \mathscr{G}_{m_2}$. Given these choices, define $s^{**} = s^* + \ell \cdot \mathsf{dist}(x^*, y^*)$. Then let

$$\begin{aligned}&\eta[y^*]_s = \left\{y \in \T_n(y^*): (y^*,s^{**})\;\lra\;(y,s) \text{ inside } \T_n(y^*)\right\},\qquad s^{**}\leq s \leq t\;\\
&E = \{V=a;\;(x^*,s^*) \;\lra\; (y^*,s^{**}) \text{ inside } \T_{n,m_1};\;\eta[y^*]_t \in \mathscr{G}_{m_2}\}.
\end{aligned}$$

Since $0 \leq s^* \leq m_2^{1/2}/8$, we have $0 \leq s^{**} \leq m_2^{1/2}/8 + \ell \cdot \mathsf{dist}(x^*, y^*) \leq m_2^{1/2}/8 + 2\ell m_1 \leq m_2^{1/2}/4$, then $m_2^{1/2}/4 \leq t - s^{**} \leq m_2^{1/2}$. By Lemma \ref{lem:ss2} and (\ref{eq:prop3}), we get
$$\P[E\;|\;V=a] \geq (\bar c \theta^{2m_1})\cdot (\sigma/2).$$
Since $\Gamma(\xi
_t) \geq V + I_E$, this completes the proof in the case $1 \leq a  <  h_4(n)$. The case $a = h_4(n)$ is the same, except that we only use $\Gamma(\xi_t) \geq V$, and the case $a \leq 0$ is trivial, so the proof is complete.
\end{proof}

\begin{lemma}\label{lem:rwest}
If $n$ is large enough, then 
\begin{align}
&\P\Ll[\left.Z^{(n)}_i \leq (3/4)h_4(n) \;\right|\; Z^{(n)}_0 = h_4(n)\Rr] \leq e^{-h_4(n)} \quad \forall i \in [0,\; e^{h_5(n)}];\label{eq:rwest1}\\
&\P\Ll[\left.Z^{(n)}_i \leq (3/4)h_4(n) \;\right|\; Z^{(n)}_0 = a\Rr]  \leq e^{-a} \quad \forall a \in \{1, \ldots, h_4(n)\},\ i \in [e^{h_{10}(n)},\; e^{h_5(n)}].\label{eq:rwest2}
\end{align}
\end{lemma}
The proof of this lemma is deferred to the appendix.

Now let us define $\mathscr{H}_n = \{A \subset {\T_n}: \Gamma(A) > (3/4)  d^{m_1}\}.$
\begin{lemma}\label{lem:Hprop}
For $n$ large enough,
\begin{eqnarray} 
&&\label{eq:HH}\P\Ll[\xi^A_t \in \mathscr{H}_n\Rr] \geq 1 - e^{-\frac{3}{4}h_4(n)} \quad \forall A \in \mathscr{H}_n,\; t \in [m_2^{1/2},\;e^{h_5(n)}];\\
&&\label{eq:Hlong}\P\Ll[\xi^A_t \neq \varnothing,\;\xi^A_t \notin \mathscr{H}_n\Rr] \leq e^{-h_5(n)},\quad \forall A \subset {\T_n},\; t \in [e^{h_9(n)},\;e^{h_5(n)}];\\
&&\label{eq:Hxi}\P\Ll[\xi^A_{m_2^{1/2}} \cap A' = \varnothing\Rr] \leq e^{-h_4(n)/4n},\quad \forall A, A'\in \mathscr{H}_n.
\end{eqnarray}
\end{lemma}
\begin{proof}
For (\ref{eq:HH}), let $i_0$ be the largest integer such that $i_0\cdot m_2^{1/2} \leq t - m_2^{1/2}/2$. By Lemma \ref{lem:compRW}, $\left(\Gamma\left(\xi^{\underline{1}}_0\right),\; \Gamma\left(\xi^{\underline{1}}_{m_2^{1/2}}\right),\ldots,\;\Gamma\left(\xi^{\underline{1}}_{i_0 \cdot m_2^{1/2}}\right),\; \Gamma\left(\xi^{\underline{1}}_{t}\right)\right)$ is stochastically larger than $\left(Z^{(n)}_0, \ldots,\;Z^{(n)}_{i_0+1}\right) $ with $Z^{(n)}_0 \ge \frac{3}{4}h_4(n)$. Since $i_0+1 \leq 2\frac{\exp(h_5(n))}{m_2^{1/2}} \leq \exp(h_5(n))$, the result follows from (\ref{eq:rwest1}).

The same argument using (\ref{eq:rwest2}) gives
\begin{equation}
\P\Ll[\xi^A_t \in \mathscr{H}_n\Rr] > 1-e^{-a} \quad \forall A \text{ with } \Gamma(A) = a,\; t \in [e^{h_9(n)},\; e^{h_5(n)}].\label{eq:condTransp}
\end{equation}
For (\ref{eq:Hlong}), we will need (\ref{eq:condTransp}) and 
\begin{equation}\label{eq:condTrans}\P\left[\exists s \le \ell n + \frac{Sm_1}{2K} + m_2^{1/2}:\; \Gamma\left(\xi^{A}_{s}\right)  > \frac{\sigma}{4}\;\bar d^{m_1}\right]>\frac{\bar{c}\sigma}{2}\theta^{n}, \qquad \forall A \subset {\T_n},\; A \neq \varnothing. 
\end{equation}
Let us prove (\ref{eq:condTrans}). Choose $x \in A$; by Lemma \ref{lem:ss2}, with probability larger than $\bar{c}\theta^n$, we have $(x,0) \;\leftrightarrow\;(o, s_1)$ for some $s_1 \le \ell n$. Conditioned on this event, by Lemma \ref{lem:ss1}(ii.), with probability larger than $\sigma$ we have $|\xi^{A}_{s_1 + (S/2K)m_1} \cap \mathbb{L}_{n}(m_1)| >\bar d^{m_1}$. Also conditioning on this event, by (\ref{eq:prop3}), the number of $y \in \mathbb{L}_{n}(m_1)$ such that $\xi^{A}_{s_1 + \frac{Sm_1}{2K} + m_2^{1/2}} \cap \T_n(y) \in \mathscr{G}_{m_2}$ stochastically dominates a Binomial$(\bar d^{m_1},\;\sigma/2)$ random variable. If $n$ is large enough, such a random variable is larger than $(\sigma/4)\bar d^{m_1}$ with probability larger than $1/2$. This proves (\ref{eq:condTrans}).

We now turn to (\ref{eq:Hlong}). Let $\upkappa = \inf\left\{s: \Gamma(\xi^{A}_s) > \frac{\sigma}{4}\bar d^{m_1}\right\}$. For $ t \in [e^{h_8(n)},\;e^{h_5(n)}]$ we have
\begin{equation}\label{eq:auxFim}
\begin{split}\P\big[\;\xi^{A}_{t} \neq \varnothing,\;\xi^A_t \notin \mathscr{H}_n \;\big] \leq \P\Ll[\;\xi^{A}_{e^{h_9(n)}} \neq \varnothing,\;\upkappa  > e^{h_9(n)}\;\Rr] + \P\big[\;\xi^A_t \notin \mathscr{H}_n \;|\; \upkappa \leq e^{h_9(n)} \;\big].
\end{split}\end{equation}
The second term on the right-hand side is less than $\exp\left(-\frac{\sigma}{4}\bar d^{m_1} \right)<\frac{1}{2}\;e^{-h_5(n)}$ by (\ref{eq:condTransp}). Let us show that the first term on the right-hand side of (\ref{eq:auxFim}) is also smaller than $\frac{1}{2}\;e^{-h_5(n)}$. We use Lemma \ref{lem:Rep} with the following choice of parameters:
$$\A = \T_n,\quad t_0 = \ell n + \frac{S}{2K}m_1 + m_2^{1/2},\quad E = \left\{A: \Gamma(A) > \frac{\sigma}{4}\bar d^{m_1}\right\},\quad \upepsilon_0 = \frac{\bar{c}\sigma}{2} \theta^n.$$
With this choice, (\ref{eq:cond1Rep}) is exactly (\ref{eq:condTrans}), and $\{\xi_s \neq \varnothing\} \subset \{\mathcal{N}_{\T_n, t_0}(s) \geq \lfloor s/t_0 \rfloor - 1\}$ for any $s$. We then have, by (\ref{eq:repCon}),
$$\begin{aligned}\P\left[\xi_{e^{h_9(n)}} \neq \varnothing,\; \upkappa > e^{h_9(n)} \right] &\leq \P\left[\mathcal{N}_{\T_n, t_0}(e^{h_9(n)}) \ge \lfloor e^{h_9(n)}/t_0 \rfloor - 1,\; \upkappa > e^{h_9(n)} \right] \\
&\leq\exp\left(-\frac{\bar c\sigma}{2}\cdot \theta^{{n}}\cdot \left(\left \lfloor \frac{\exp(h_9(n))}{t_0}\right \rfloor -1\right) \right),\end{aligned}$$
which is of course much smaller than $e^{-h_5(n)}/2$. We have thus proved that the right-hand side of (\ref{eq:auxFim}) is smaller than $e^{-h_5(n)}$.

Finally, let us prove (\ref{eq:Hxi}). By the definition of $\Gamma$, the hypothesis gives
$$\left|\{x \in \mathbb{L}_{n}(m_1): A \cap \T_n(x) \in \mathcal{G}_{m_2},\;A' \cap \T_n(x) \in \mathcal{G}_{m_2}\} \right| \geq h_4(n)/4.$$
The result then follows from (\ref{eq:prop4}).
\end{proof}

\begin{corollary}\label{prop:allInf}
For $n$ large enough, ${\displaystyle \P\Ll[\xi^{\underline{1}}_{\exp(h_5(n))} = \varnothing\Rr] \leq e^{-h_5(n)}}$.
\end{corollary}
\begin{proof}
The fully infected configuration is in $\mathscr{H}_n$ (since $\mathscr{H}_n$ is non-empty and increasing), so the statement follows directly from (\ref{eq:HH}).
\end{proof}

\begin{corollary}
\label{prop:coupling}
For $n$ large enough,
$$\P\left[\exists x \in \T_n: \xi^{x}_{\exp(h_6(n))} \neq  \varnothing,\; \xi^{x}_{\exp(h_6(n))} \neq \xi^{\underline{1}}_{\exp(h_6(n))} \right] \leq e^{-d^{(n^2)}}.$$
\end{corollary}
\begin{proof}
We first prove that, for any $A \subset {\T_n}$,
\begin{equation}\P\left[\xi^{A}_{\exp(h_7(n))} \neq \varnothing,\;\xi^{A}_{\exp(h_7(n))} \neq \xi^{\underline{1}}_{\exp(h_7(n))} \right] \leq e^{-h_7(n)}.\label{eq:partCoup}\end{equation}
To this end, write $r_n = \exp(h_7(n)),\;r_n' = r_n/2,\; r_n'' = r_n' - m_2^{1/2}$. Fix $y \in \T_n$; we have
$$\begin{aligned}
\P\left[\xi^{A}_{r_n} \neq \varnothing,\; \xi^{A}_{r_n}(y) \neq \xi^{\underline{1}}_{r_n}(y) \right] &\leq \P\left[\xi^{A}_{r_n''} \neq \varnothing,\; \xi^{A}_{r_n''}\notin \mathscr{H}_n \right] +\P\left[\hat \xi^{(y, r_n)}_{r_n'} \neq \varnothing,\; \hat \xi^{(y, r_n)}_{r_n'}\notin \mathscr{H}_n \right] \\&\qquad\qquad\qquad\qquad\qquad+ \P\left[\xi^{A}_{r_n'} \cap \hat \xi^{(y,r_n)}_{r_n'} = \varnothing \left|\;  \xi^{A}_{r_n''}, \hat \xi^{(y, r_n)}_{r_n'}\in \mathscr{H}_n\right.\right]
\end{aligned}$$
The first and second terms are less than $e^{-h_5(n)}$ by (\ref{eq:Hlong}). The third term is less than $e^{-h_4(n)/4n}$ by (\ref{eq:Hxi}). Summing over all choices of $y$, we conclude that the probability in (\ref{eq:partCoup}) is less than $d^{n+1}\cdot (2e^{-h_5(n)} + e^{-h_4(n)/4n}) < e^{-h_7(n)}$. 

Now, the probability in the statement of the corollary is less than
$$\begin{aligned}&\sum_{x\in\T_n}\sum_{\substack{B \subset \T_n:B \neq \varnothing}} \P\left[ \xi^x_{e^{h_6(n)} - e^{h_7(n)}} = B \neq \xi^{\un{1}}_{e^{h_6(n)} - e^{h_7(n)}}\right] \cdot \P\left[\xi^B_{e^{h_7(n)}} \neq \varnothing,\;\xi^B_{e^{h_7(n)}} \neq \xi^{\un{1}}_{e^{h_7(n)}}\right] \\&\qquad\leq e^{-h_7(n)}\cdot \sum_{x\in\T_n}\P\left[ \xi^x_{e^{h_6(n)}-e^{h_7(n)}} \neq \varnothing,\;\xi^x_{e^{h_6(n)}-e^{h_7(n)}} \neq \xi^{\un{1}}_{e^{h_6(n)}-e^{h_7(n)}} \right].\end{aligned}$$
Iterating, this shows that
$$\P\left[\exists x \in \T_n: \xi^{x}_{\exp(h_6(n))} \neq  \varnothing,\; \xi^{x}_{\exp(h_6(n))} \neq \xi^{\underline{1}}_{\exp(h_6(n))} \right] \leq |\T_n|\cdot (e^{-h_7(n)})^{e^{h_6(n)}/e^{h_7(n)}} < e^{-d^{(n^2)}}$$
when $n$ is large enough.
\end{proof}

\begin{corollary}\label{cor:anyCouple}
For $n$ large enough,
$$\P\Ll[\xi^o_{\exp(h_5(n))} \neq \varnothing \Rr] > \frac{\sigma}{4}.$$
\end{corollary}
\begin{proof}
As explained in the proof of (\ref{eq:condTrans}), we have
$$\P\Ll[\exists t: \Gamma(\xi^o_t) > \frac{\sigma}{4}\bar{d}^{m_1} \Rr] > \frac{\sigma}{2}.$$
If $\Gamma(\xi^o_t) > \frac{\sigma}{4}\bar{d}^{m_1}$ for some $t$, then (\ref{eq:condTransp}) guarantees that, with probability larger than $1 - e^{-(\sigma/4)\bar{d}^{m_1}}$, we have $\xi^{o}_{t+e^{h_5(n)}} \in \mathscr{H}_n$, so in particular, $\xi^o_{e^{h_5(n)}} \neq \varnothing$.
\end{proof}

\subsection{Level 4: survival for time $e^{cd^n}$}
\label{ss:l4}
We start stating a simple result about the extinction time of the contact process. We refer the reader to \cite[Lemma~4.5]{mmvy} for the proof.
\begin{lemma}
\label{l:attract}
For every $s > 0$, we have
$$
\P\Ll[\uptau_{\T_n} \le s \Rr] \le \frac{s}{\E[\uptau_{\T_n}]},
$$	
Moreover, there exists a constant $C$ such that for every $n$, $\E[\uptau_{\T_n}] \le e^{C|\T_n|}$.
\end{lemma}

We will write $(\xi^{A,s}_t)_{t \ge s}$ for the contact process started at time $s$ with $A$ infected, that is,
$$\xi^{A,s}_t = \{y: A \times \{s\} \;\leftrightarrow \; (y,t)\},\qquad t \ge s.$$ 
Similarly we write $\xi^{x,s}_t$ and $\xi^{\un{1},s}_t$. We of course assume these processes are defined with the same graphical construction as the one used for the definition of the original contact process $(\xi_t)_{t \ge 0}$ on $\T_n$, so that we can consider them all in the same probability space. 

Let $t_k = k\cdot e^{h_6(n)},\; k = 0, 1, \ldots$ and define the events
$$\begin{aligned}
&E_k = \left\{\xi^{\un{1},t_{k-1}}_{t_{k+1}} \neq \varnothing \right\}, \\
&F_k = \left\{\text{for all } x \in \mathbb{T}_n,\text{ either } \xi^{x,t_{k-1}}_{t_k} = \varnothing \text{ or }\xi^{x, t_{k-1}}_{t_k} = \xi^{\un{1},t_{k-1}}_{t_k}\right\},\qquad k\geq 1.
\end{aligned}$$
\begin{lemma}
\label{lem:NemLemC}
On $\cap_{\ell = 1}^k (E_\ell \cap F_\ell)$, we have $\xi^{\un{1}}_{t_{k+1}} \neq \varnothing$.
\end{lemma}
\begin{proof}
It is enough to prove that, for any $k$,
\begin{equation}\left\{\xi^{\un{1}}_{t_k} \neq \varnothing\right\} \cap E_k \cap F_k \subseteq \left\{ \xi^{\un{1}}_{t_{k+1}} \neq \varnothing\right\}\label{eq:newk}\end{equation}
For $k = 0$, this follows directly from the definition of $E_0$. Assume $k \geq 1$. Writing
$$\xi^{\un{1}}_{t_k} = \bigcup_{x \in \xi^{\un{1}}_{t_{k-1}}} \xi^{x, t_{k-1}}_{t_k},\qquad \xi^{\un{1},t_{k-1}}_{t_{k+1}} = \bigcup_{y\in\mathbb{T}_n} \xi^{y,t_{k-1}}_{t_{k+1}},$$
we see that the occurrence of $\{\xi^{\un{1}}_{t_k} \neq \varnothing\}$ and $E_k=\{\xi^{\un{1},t_{k-1}}_{t_{k+1}} \neq \varnothing\}$ imply that there exist $x \in \xi^{\un{1}}_{t_{k-1}}$, $y \in \mathbb{T}_n$ such that $\xi^{x,t_{k-1}}_{t_k},\;\xi^{y,t_{k-1}}_{t_k},\;\xi^{y,t_{k-1}}_{t_{k+1}}$ are all non-empty. Since $F_k$ occurs, $\xi^{x,t_{k-1}}_{t_k},\;\xi^{y,t_{k-1}}_{t_k}$ being non-empty implies that they are equal (as both are equal to $\xi^{\un{1},t_{k-1}}_{t_{k}}$), hence we also have
$$\varnothing \neq \xi^{y,t_{k-1}}_{t_{k+1}}=\xi^{x,t_{k-1}}_{t_{k+1}} \leq \xi^{\un{1}}_{t_{k+1}}.$$
\end{proof}

Our final (Level 4) recursion will be very simple. Our subtrees will be the $d$ trees that are rooted at the neighbours of the root; we write $x_1, \ldots, x_d$ to denote these neighbours. 
\begin{proposition}
\label{p:coupling}
For $n$ large enough, we have
$$
\E[\uptau_{\T_{n}}] \ge \left(\frac{\E[\uptau_{\T_{n-1}}]}{e^{h_6(n)}}\right)^d.
$$
\end{proposition}
\begin{proof}
By the above lemma,
\begin{equation}
\label{eq:newPC}
\P\left[ \uptau_{\T_n} \leq t\right] \leq \sum_{k=0}^{\lfloor t/e^{h_6(n)} \rfloor} \left( \P[E_k^c] + \P[F_k^c] \right).
\end{equation}
We have
\begin{equation}
\label{eq:newPC1}
\P[E_k^c] \leq \prod_{i=1}^d \P[\T_n(x_i)\times \{t_{k-1}\} \nleftrightarrow \T_n(x_i)\times \{t_{k+1}\} \text{ inside } \T_n(x_i)] \leq \left( \frac{2e^{h_6(n)}}{\E[\uptau_{\T_{n-1}}]}\right)^d,
\end{equation}
where the second inequality follows from Lemma \ref{l:attract}. We also have, by the definition of $F_k$ and Corollary \ref{prop:coupling},
\begin{equation}
\label{eq:newPC2}
\P[F_k^c] \leq e^{-d^{(n^2)}}.
\end{equation}
Using (\ref{eq:newPC1}) and (\ref{eq:newPC2}) in (\ref{eq:newPC}) with $t = 2\left(\frac{\E[\uptau_{\T_n}]}{e^{h_6(n)}}\right)^d$ we get:
$$\P\left[\uptau_{\T_n} \le 2\left(\frac{\E[\uptau_{\T_{n-1}}]}{e^{h_6(n)}}\right)^d \right] \le 2\left( \frac{\E[\uptau_{\T_{n-1}}]}{e^{h_6(n)}}\right)^d \cdot \frac{1}{e^{h_6(n)}} \cdot \left( \left( \frac{2e^{h_6(n)}}{\E[\uptau_{\T_{n-1}}]} \right)^d + e^{-d^{(n^2)}}\right) \le \frac{1}{2}$$
when $n$ is large. This finishes the proof.
\end{proof}

\begin{proofof}\emph{Theorem~\ref{thm:mainsup}}.
From Proposition~\ref{p:coupling}, we know that
$$
\frac{\log \E[\uptau_{\T_{n}}]}{d^{n}} + \frac{h_6(n)}{d^{n-1}} \ge \frac{\log \E[\uptau_{\T_{n-1}}]}{d^{n-1}}.
$$
Corollary~\ref{prop:allInf} ensures that for $n$ sufficiently large,
$$
\log \E[\uptau_{\T_{n}}] \ge h_5(n)/2.
$$
Writing 
$$
u_n = \frac{\log \E[\uptau_{\T_{n}}]}{d^n},
$$
we thus get that for $n$ sufficiently large,
$$
u_{n} \Ll(1 + \frac{1}{n^2}\Rr) \ge u_{n-1} .
$$
Letting $\rho_n = \prod_{i = 1}^n (1+1/i^2)$, we can rewrite this as
$u_n \ \rho_n \ge u_{n-1} \ \rho_{n-1}$, for $n$ sufficiently large. In other words, the sequence $u\rho$ is ultimately increasing. It thus converges to some constant $c \in \R \cup \{+\infty\}$. Clearly, $u_n \ \rho_n$ is positive if $n$ is large enough, so $c > 0$. Since the sequence $\rho$ converges to a finite constant, it follows from Lemma~\ref{l:attract} that $c$ is finite. We have thus shown that the sequence $u\rho$ converges to $c \in (0,+\infty)$, and this implies part (a) of the theorem. Part (b) now follows from Proposition A.1 in \cite{mmvy}.
\end{proofof}

\subsection{Other initial configurations}
In this subsection we will prove Theorem \ref{thm:mainsupp}. We start proving that the theorem follows from:
\begin{proposition}
\label{prop:auxEnd}There exists $\delta >0$ such that, for large enough $n$,
$$\inf_{A \subset \T_n,\; A \neq \varnothing}\;\P\Ll[\uptau_A > \exp(h_5(n)) \Rr] > \delta.$$
\end{proposition}

\begin{proofof}{\textit{Theorem \ref{thm:mainsupp}}}.
Since, for any $A$,
$$\left\{ \uptau_A > \upalpha \E[\uptau_{\T_n}]\right\} \supset \left\{\xi^A_{e^{h_5(n)}} = \xi^{\underline{1}}_{e^{h_5(n)}} \neq \varnothing,\; \xi^{\underline{1}}_{\upalpha \E[\uptau_{\T_n}]} = \xi^{\underline{1},e^{h_5(n)}}_{\upalpha \E[\uptau_{\T_n}]} \neq \varnothing\right\},$$
we get
$$\begin{aligned}\P\Ll[\xi^A_{\upalpha\E[\uptau_{\T_n}]} \neq \varnothing\Rr] &\ge \P\Ll[\xi^A_{e^{h_5(n)}} \neq \varnothing\Rr]\cdot \P\Ll[\xi^{\underline{1},e^{h_5(n)}}_{\upalpha \E[\uptau_{\T_n}]} \neq \varnothing\Rr]\\&\quad - \P\Ll[\xi^A_{e^{h_5(n)}} \neq \varnothing,\;\xi^A_{e^{h_5(n)}} \neq \xi^{\underline{1}}_{e^{h_5(n)}}\Rr] - \P\Ll[ \xi^{\underline{1}}_{\upalpha \E[\uptau_{\T_n}]} \neq \xi^{\underline{1},e^{h_5(n)}}_{\upalpha \E[\uptau_{\T_n}]}\Rr].\end{aligned}$$
Now note that, by Theorem \ref{thm:mainsup}, 
$$\lim_{n \to \infty} \P\Ll[\xi^{\underline{1},e^{h_5(n)}}_{\upalpha \E[\uptau_{\T_n}]} = \varnothing \Rr] = \lim_{n \to \infty} \P\Ll[ \uptau_{\T_n} > \upalpha \E[\uptau_{\T_n}] - e^{h_5(n)}\Rr] = e^{-\upalpha}$$
and, by Corollary \ref{prop:coupling}, $$\lim_{n\to\infty}\P\Ll[\xi^A_{e^{h_5(n)}} \neq \varnothing,\;\xi^A_{e^{h_5(n)}} \neq \xi^{\underline{1}}_{e^{h_5(n)}}\Rr]= \lim_{n\to\infty}\P\left[\xi^{\underline{1}}_{\upalpha \E[\uptau_{\T_n}]} \neq \xi^{\underline{1}, e^{h_5(n)}}_{\upalpha \E[\uptau_{\T_n}]}\right] = 0,$$
so the desired statement follows.
\end{proofof}

We now turn to the proof of Proposition \ref{prop:auxEnd}. We will need the following preliminary result.
\begin{lemma}\label{lem:taux}
If $n$ is large enough, $x \in \T_n$ and $t \ge \exp(h_6(n))$,
$$\P\Ll[\uptau_{\{x\}} \in [\exp(h_6(n)),\; t]\Rr] \le e^{-d^{(n^2)}} + \frac{t}{\E[\uptau_{\T_n}]}.$$
\end{lemma}
\begin{proof}
The left-hand side is bounded by $$\P\Ll[\xi^x_{e^{h_6(n)}} \neq \varnothing,\; \xi^x_{e^{h_6(n)}} \neq \xi^{\underline{1}}_{e^{h_6(n)}}\Rr] + \P\Ll[\xi^{\underline{1}}_t = \varnothing\Rr].$$
The first term is less than $e^{-d^{(n^2)}}$ by Corollary \ref{prop:coupling} and the second term is less than $\frac{t}{\E[\uptau_{\T_n}]}$ by Lemma \ref{l:attract}.
\end{proof}

\begin{proofof}\textit{Proposition \ref{prop:auxEnd}}.
We choose $N$ large enough that
\begin{itemize}
\item[(1)] $h_5(n-1) > h_6(n) \; \forall n \ge N;$
\item[(2)] $\text{the conclusion of Corollary \ref{cor:anyCouple} holds for all $n \ge N$}$;
\item[(3)] ${\displaystyle \delta := \min(\sigma/4,\; e^{-h_5(N)}) - \sum_{j = N}^\infty e^{-d^{j/2}} > 0}.$
\end{itemize}

Now fix $n \ge N$ and let $A \subset \T_n$ be non-empty. Fix $x \in A$. If $\T_n(x)$ has height at least $N$ (or, in other words, of $\mathsf{dist}(o,x) \le n-N$), then let $y_0 = x$. Otherwise, let $y_0$ be the point in the path from $x$ to $o$ which is at distance $n-N$ from $o$. Then let $y_1 \sim y_2 \sim \ldots \sim y_k = o$ be the vertices in the path from $y_0$ to $o$, so that $y_{i+1} = \mathsf{p}(y_i)$ for each $i$ and $k = \mathsf{dist}(o,y_0)$. For each $i$, let $j_i$ be the height of $\T_n(y_i)$, that is, $j_i = n - \mathsf{dist}(o,y_i)$.

As before, we assume given a Harris system $H$ for the contact process on $\T_n$. For each $i$, we write $$\xi^x_{i,t}(z) = I\{(x,0) \;\lra\;(z,t)\;\text{inside } \T_n(y_i)\},\qquad z \in \T_n(y_i),\;t\ge 0.$$
Simply put, $(\xi^x_{i,t})_{t\ge 0}$ is the contact process on $\T_n(y_i)$, started from only $x$ infected, and constructed using the restriction of $H$ to $\T_n(y_i)$. Then define, for each $i$, the times
$$
\upkappa_i = \sup\{t: \xi^x_{i,t} \neq \varnothing\}.
$$
Also define the events
$$\begin{aligned}
&E_0 = \left\{\upkappa_0 > e^{h_5(j_0)} \right\};\\
&E_i = \left\{\upkappa_i \notin [e^{h_6(j_i)},\; e^{h_5(j_i)}]\right\},\qquad 1 \le i \le k
\end{aligned}$$ 
(obviously, if $x = y_0 = o$ and thus $k = 0$, the second line should be ignored).

We now claim that \begin{equation}\label{eq:ksGood}E_0 \cap \cdots \cap E_k \subset \{\xi^x_{\exp(h_5(n))} \neq \varnothing\} \subset \{\xi^A_{\exp(h_5(n))} \neq \varnothing\}.\end{equation}
Indeed, the second inclusion is evident and the first one is verified using the fact that $\upkappa_{i+1} \ge \upkappa_i$ for each $i$, together with item (1) in the choice of $N$:
$$\begin{aligned}
E_0 \cap E_1 \cap \cdots \cap E_k &= \left\{ \upkappa_0 > e^{h_5(j_0)}\right\} \cap \left\{\upkappa_1 \notin [e^{h_6(j_1)},\;e^{h_5(j_1)}]\right\} \cap E_2 \cap \cdots \cap E_k\\
&\subset \left\{\upkappa_1 > e^{h_5(j_1)}\right\} \cap E_2 \cap \cdots \cap E_k
\end{aligned}$$
and iterate.

We will thus be finished if we show that 
\begin{equation}\P\Ll[E_0 \cap \cdots \cap E_k\Rr] > \delta. \label{eq:toEnd}\end{equation}
First note that
$$\P[E_0] \ge \min(\sigma/4,\;e^{-h_5(N)}).$$
Indeed, if $j_0 > N$, then $P[E_0] > \sigma/4$ by Corollary \ref{cor:anyCouple} and if $j_0 = N$, then $P[E_0] > e^{-h_5(N)}$ simply by the fact that this is the probability that $x$ does not recover from time 0 to time $e^{h_5(N)}$. Second, note that for each $i \ge 1$, by Lemma \ref{lem:taux} and Theorem \ref{thm:mainsup},
$$\P[E_i] \ge 1 - e^{-d^{(j_i)^2}} - \frac{e^{h_5(j_i)}}{\E[\uptau_{\T_{j_i}}]} >  1 - e^{-d^{j_i/2}}.$$
(\ref{eq:toEnd}) now follows from item (3) in the choice of $N$.
\end{proofof}

\appendix

\section{Appendix -- Random walk estimates}
The purpose of this appendix is to prove Lemma~\ref{lem:rwest}, which is a statement concerning the Markov chain $Z^{(n)}$ with transition probabilities defined in \eqref{def:Z}. We begin with a statement concerning hitting times.
\begin{lemma}
\label{l:hitting}
For the chain ${Z}^{(n)}$ with transition probabilities defined in \eqref{def:Z}, let $H_0$ denote the hitting time of $\Z_- = \Z \cap (-\infty,0]$ and $H_1$ denote the hitting time of $h_4(n)$. For every $n$ large enough and every $a \le h_4(n)$,
\begin{equation}
\label{hitting}
\P\Ll[ H_0 < H_1 \ | \ {Z}^{(n)}_0 = a \Rr] \le e^{-a n^{2/7}}.
\end{equation}
Let $H_{3/4}$ be the hitting time of $\Z \cap (-\infty,(3/4)h_4(n)]$. For every $n$ large enough and every $a \le h_4(n)$,
\begin{equation}
\label{hitting2}
\P\Ll[ H_{3/4} < H_1 \ | \ {Z}^{(n)}_0 = a \Rr] \le e^{-(a - (3/4)h_4(n)) n^{2/7}}.
\end{equation}
\end{lemma}
\begin{proof}
Recall that the transition probabilities of $Z^{(n)}$ in \eqref{def:Z} are defined in terms of $p^{(n)}$, the probability mass function of a Binomial($h_4(n),e^{-n^{1/3}}$). Clearly, a Bernoulli random variable with parameter $p$ is stochastically dominated by a Poisson random variable with parameter $-\log(1-p)$. Hence, a Binomial($h_4(n),e^{-n^{1/3}}$) is stochastically dominated by a Poisson random variable with parameter 
$$
-h_4(n) \log(1-e^{-n^{1/3}}) \le e^{-2n^{2/7}} =: \uplambda_n,
$$
provided $n$ is sufficiently large.
Let $\ov{p}^{(n)}$ be the probability mass function of a Poisson random variable with parameter $\uplambda_n$, and consider the Markov chain $\ov{Z}^{(n)}$ whose transition probabilities are defined by
$$
\begin{array}{l} 
\underline{\text{for } k < h_4(n):}\medskip\\ P(k, k+1) = \theta^{4m_1};\medskip\\
P(k,k) = \ov{p}^{(n)}(0) - \theta^{4m_1};\medskip\\
P(k, k-a) = \ov{p}^{(n)}(a),\; a > 0;
\end{array} \qquad 
\begin{array}{l}\underline{\text{for }k = h_4(n):}\medskip\\ P(k, k-a) = \ov{p}^{(n)}(a),\; a \ge 0;\;\\[+0.35cm] \; \\ \; \\
\end{array}
$$
Compared with the definition of $Z^{(n)}$, we replaced $p^{(n)}$ by $\ov{p}^{(n)}$ for the jumps to the left, decreased the probability to jump to the right (for $n$ sufficiently large), and changed the definition of the Markov chain over $\Z_-$ for convenience. Clearly, $Z^{(n)}$ is stochatically dominated by $\ov{Z}^{(n)}$ until reaching $\Z_-$, so it suffices to prove Lemma~\ref{l:hitting} with $\ov{Z}^{(n)}$ in place of ${Z}^{(n)}$.

Let us write $\mcl{L}$ for the generator of the Markov chain $\ov{Z}^{(n)}$, that is,
$$
\mcl{L}f(x) = \theta^{4m_1} (f(x+1) - f(x)) + e^{-\uplambda_n} \sum_{a = 1}^{+\infty}  \frac{\uplambda_n^a}{a!}(f(x-a)-f(x))    \qquad (x < h_4(n)).
$$
Let $f(a)$ be the left-hand side of \eqref{hitting}, and $\td{f}(a) = e^{-an^{2/7}}$. For $a \in \Z \cap (0,h_4(n))$, we have $\mcl{L} f(a) = 0$. On the other hand, for such $a$, we have
$$
\mcl{L}\td{f}(x) = \theta^{4m_1} (e^{-n^{2/7}} -1)\td{f}(x) + e^{-\uplambda_n} \sum_{a = 1}^{+\infty} \frac{\uplambda_n^a}{a!}(e^{an^{2/7}} - 1) \td{f}(x).
$$
Recalling that $\uplambda_n = e^{-2n^{2/7}}$, we get
$$
e^{-\uplambda_n} \sum_{a = 1}^{+\infty} \frac{\uplambda_n^a}{a!}(e^{an^{2/7}} - 1)  \le \sum_{a = 1}^{+\infty} \frac{e^{-an^{2/7}}}{a!} = \exp\Ll(e^{-n^{2/7}}\Rr) - 1,
$$
so that
$$
\mcl{L}\td{f}(x) \le \Ll[ \theta^{4\lfloor n^{1/4} \rfloor}\Ll( e^{-n^{2/7}} -1 \Rr) + \exp\Ll(e^{-n^{2/7}}\Rr) - 1 \Rr] \td{f}(x).
$$
The square brackets above behave like $-\theta^{4\lfloor n^{1/4} \rfloor}$ to leading order. In particular, $\mcl{L}\td{f}(x) \le 0$ for all large enough $n$. As a consequence, $\mcl{L}(f - \td{f}) \ge 0$ on $\Z \cap (0,h_4(n))$. By the maximum principle,
$$
\max_{\Z \cap (0,h_4(n))} (f-\td{f}) \le \max_{\Z_- \cup \{h_4(n)\}} (f-\td{f}) = 0,
$$
and this proves \eqref{hitting}. The proof of \eqref{hitting2} is identical.
\end{proof}

\begin{proofof}\textit{Lemma~\ref{lem:rwest}}.
We will actually show the stronger statements that
\begin{align}
&\P\Ll[\inf_{i \in [0,\; e^{h_5(n)}]}\left.Z^{(n)}_i \leq (3/4)h_4(n) \;\right|\; Z^{(n)}_0 = h_4(n)\Rr] \leq e^{-h_4(n)} \label{eq:rwest1s}\\
&\P\Ll[\inf_{i \in [e^{h_8(n)},\; e^{h_5(n)}]}\left.Z^{(n)}_i \leq (3/4)h_4(n) \;\right|\; Z^{(n)}_0 = a\Rr]  \leq e^{-a} \quad \forall a \in \{1, \ldots, h_4(n)\}.\label{eq:rwest2s}
\end{align}
Let us consider \eqref{eq:rwest1s} first. We can decompose the trajectory into a sequence of excursions from $h_4(n)$. Until time $e^{h_5(n)}$, there can happen no more than $e^{h_5(n)}$ excursions. Among $e^{h_5(n)}$ excursions, the probability that at least one excursion starts with a jump to a point below $(7/8) h_4(n)$ is smaller than
$$
e^{h_5(n)} \sum_{a = (1/8)h_4(n)}^{+\infty} p^{(n)}(a).
$$
This is smaller than the same quantity with $p^{(n)}$ replaced by $\ov{p}^{(n)}$, where $\ov{p}^{(n)}$ was introduced in the proof of Lemma~\ref{l:hitting}, that is,
$$
e^{h_5(n)} e^{-\uplambda_n} \sum_{a = (1/8)h_4(n)}^{+\infty} \frac{\uplambda_n^a}{a!} \le e^{h_5(n)-\uplambda_n} \lambda_n^{(1/8)h_4(n)}\sum_{a = 0}^{+\infty} \frac{\uplambda_n^a}{a!} =  e^{h_5(n)} \lambda_n^{(1/8)h_4(n)}.
$$
Using the fact that $h_4(n)$ is much larger than $h_5(n)$ while $\uplambda_n$ tends to $0$, we get that the quantity above is much smaller than
$$
\uplambda_n^{(1/16) h_4(n)} \le e^{-2h_4(n)}
$$
for $n$ sufficiently large. 

Outside of this event of very small probability, we know that each of the first $e^{h_5(n)}$ excursions starts with a jump to a point above $(7/8) h_4(n)$. By \eqref{hitting2}, we know that starting from such a point, the probability to reach $(3/4) h_4(n)$ before $h_4(n)$ is smaller than $e^{-(1/8)h_4(n) n^{2/7}}$.
Since
$$
e^{h_5(n)} e^{-(1/8)h_4(n) n^{2/7}} \le e^{-2h_4(n)},
$$
we have completed the proof of \eqref{eq:rwest1s}.

For the proof of \eqref{eq:rwest2s}, observe first that the walk moves on a time scale of order $\theta^{-2\lfloor n^{1/4} \rfloor}$, which is much smaller than $e^{h_8(n)}$ (more precisely, we note that $e^{h_8(n)}/ \theta^{-2\lfloor n^{1/4} \rfloor} \gg h_4(n)$), so one can check that outside of an event whose probability is much smaller than $e^{-h_4(n)}$, the walk reaches $\{h_4(n)\} \cup \Z_-$ before time $e^{h_8(n)}$. Moreover, by \eqref{hitting}, the probability to reach $\Z_-$ before reaching $h_4(n)$ is bounded by $e^{-an^{2/7}}$.  Once $h_4(n)$ is reached, we can use the Markov property and the reasoning leading to \eqref{eq:rwest1s} to conclude.
\end{proofof}

\end{document}